\documentclass[12pt]{amsart}
\usepackage[utf8]{inputenc}
\usepackage[english]{babel}

\usepackage{amsmath,amsthm,amsfonts,mathrsfs,amssymb}
\usepackage{bm}
\usepackage{cite}
\usepackage{hyperref}

\newtheorem{theorem}{Theorem}[section]
\newtheorem{lemma}[theorem]{Lemma}
\newtheorem{corollary}[theorem]{Corollary}
\newtheorem{remark}[theorem]{Remark}
\newtheorem{proposition}[theorem]{Proposition}
\newtheorem{definition}[theorem]{Definition}

\newenvironment{proposition*}[1]{\smallskip\noindent{\bf #1.}\it}{\medskip}
\newenvironment{theorem*}[1]{\smallskip\noindent{\bf #1.}\it}{\medskip}
\newenvironment{proofof}[1]{\smallskip\noindent{\it #1}\rm}
                {\hspace*{\fill} $\Box$\medskip}

\numberwithin{equation}{section}
\newcommand{\slim}{\operatornamewithlimits{s-lim}}

\newcommand\opn{\operatorname}
\newcommand\diag{\operatorname{diag}}
\newcommand\dom{\operatorname{dom}}

\newcommand\supp{\operatorname{supp}}
\newcommand\myRe{\operatorname{Re}}
\newcommand\myIm{\operatorname{Im}}

\newcommand\wt{\widetilde}

\newcommand{\bla}{\bm{\lambda}}

\newcommand{\eps}{\varepsilon}

\newcommand{\bC}{\mathbb{C}}
\newcommand{\bN}{\mathbb{N}}
\newcommand{\bR}{\mathbb{R}}

\newcommand{\cM}{\mathcal{M}}
\newcommand{\cQ}{\mathcal{Q}}

\newcommand{\bkappa}{\bm{\varkappa}}

\begin{document}

\title[Integrable reflectionless potentials]%
{Inverse scattering for reflectionless Schr\"odinger operators with integrable potentials and generalized soliton solutions for the KdV equation}%
\author[R.~Hryniv  \and B.~Melnyk \and Ya.~Mykytyuk]%
{Rostyslav Hryniv, Bohdan Melnyk, \and Yaroslav Mykytyuk}%

\address[R.H.]{
	Ukrainian Catholic University, 2a Kozelnytska str., 79026, Lviv, Ukraine \and
	University of Rzesz\'{o}w, 1 Pigonia str., 35-310 Rzesz\'{o}w, Poland
}%
\email{rhryniv@ucu.edu.ua, rhryniv@ur.edu.pl}

\address[B.M.]{Ivan Franko National University of Lviv,
	1 Universytetska st., 79602 Lviv, Ukraine}%
\email{bohdmelnyk@gmail.com}%

\address[Ya.M.]{Ivan Franko National University of Lviv,
	1 Universytetska st., 79602 Lviv, Ukraine}%
\email{yamykytyuk@yahoo.com}%

\thanks{}%

\date{\today}%

\begin{abstract}
We give a complete characterisation of the reflectionless Schr\"odinger operators on the line with integrable potentials, solve the inverse scattering problem of reconstructing such potentials from the eigenvalues and norming constants, and derive the corresponding generalized soliton solutions of the Korteweg--de Vries equation.
\end{abstract}

\maketitle


\section{Introduction}

The main aim of the paper is to complete the theory of reflectionless Schr\"odinger operators
\[
	T_q := - \frac{d^2}{dx^2} + q(x)
\]
on the line with integrable potentials~$q$ that was developed in the work of Marchenko~\cite{Mar91} and Gesztesy, Karwowsky and Zhao~\cite{GKZ}. 

Reflectionless Schr\"odinger operators on the half-line were constructed for the first time by Bargmann~\cite{Bargmann} in~1949 as non-uniqueness examples in the inverse scattering problem of reconstructing the potential from the phase shift function. After the inverse scattering theory for Schr\"odinger operators on the whole line was developed by Marchenko, Gelfand and Levitan a.o.~\cite{Mar50,Mar52,GelLev} and the role of the bound states in the reconstruction was understood, Kay and Moses~\cite{KayMos} described explicitly all classical reflectionless potentials, sometimes called the Bargmann potentials. Namely, each such potential~$q$ is uniquely characterised by two sets of positive numbers, $\varkappa_1 > \varkappa_2 > \dots > \varkappa_n >0$ and $m_1, m_2, \dots, m_n$ via \begin{equation}\label{eq:int.q-refl}
	q(x) = -2 \frac{d^2}{dx^2} {\log { \det 
				{ \left( \delta_{kj} + \frac {m^2_k e^{-(\varkappa_k +\varkappa_j)x}}{\varkappa_k +\varkappa_j} \right)_{1\le k,j\le n}}}},
\qquad x\in \bR.
\end{equation}
In the above formula, $n \in \bN$ is arbitrary, $\delta_{kj}$ is the Kronecker delta, and the corresponding Schr\"odinger operator $T_q$ has spectrum consisting of the absolutely continuous part covering the positive half-line and of $n$ negative eigenvalues, $-\varkappa_1^2 < -\varkappa_2^2 < \dots < -\varkappa_n^2$ with $m_1, m_2, \dots, m_n$ being the corresponding norming constants.

The interest in such potentials was revived in 1967 after Green, Gardner, Kruskal and Miura~\cite{GGKM} suggested in 1967 a method of solving the nonlinear Korteweg--de Vries (KdV) equation
\begin{equation}\label{eq:KdV}
	u_t - 6uu_x + u_{xxx} = 0
\end{equation}
based on the inverse scattering transform for the related time-dependent family of Schr\"odinger operators
\[
	T_{u(\cdot, t)} = - \frac{d^2}{dx^2} + u (\,\cdot\,t).
\]
The main observation of~\cite{GGKM} was that the scattering data for $T_{u(\cdot, t)}$---the reflection coefficient, bound states, and norming constants---change in time $t$ in a very simple way, which allows their determination for any $t$ from their initial values and then solving the inverse scattering problem to find $u(\,\cdot\,,t)$. When the Cauchy initial value $u(\cdot,0)$ for the KdV equation~\eqref{eq:KdV} is reflectionless and is given by~\eqref{eq:int.q-refl}, one obtains the explicit $n$-\emph{soliton solution}
\begin{equation}\label{eq:int.soliton}
	u(x,t) :=  -2 \frac{d^2}{dx^2} {\log { \det 
			{ \biggl( \delta_{kj} + \frac {m^2_k e^{8\varkappa_k^2 t}e^{-(\varkappa_k +\varkappa_j)x}}{\varkappa_k +\varkappa_j} \biggr)_{1\le k,j\le n}}}}.
\end{equation}
This solution represents $n$ solitary waves first observed by Russell in 1834 that have many intriguing properties and have been the object of thoroughly study since then; see, e.g.~\cite{GH03, FadTak, Tao}. 

In early 1990-ies, Marchenko~\cite{Mar91} and Gesztesy a.o.~\cite{GKZ} suggested two different methods of constructing non-classical reflectionless potentials and corresponding generalized soliton solutions of the KdV equation. For each $\mu>0$, let $B(-\mu^2)$ be the set of all classical reflectionless potentials $q$ for which the corresponding Schr\"odinger operators~$T_q$ are bounded below by $-\mu^2$. Marchenko considered the closure $\overline{B(-\mu^2)}$ of $B(-\mu^2)$ in the topology of uniform convergence on compact subsets of $\bR$, studied properties of Schr\"odinger operators~$T_q$ with potentials~$q$ in~$\overline{B(-\mu^2)}$, and explained why such potentials can be considered reflectionless. Namely, for $T_q$ with~$q \in  \overline{B(-\mu^2)}$, the Weyl--Titchmarsh $m$-functions $m_\pm$ satisfy the relation
\begin{equation}\label{eq:int.m-refl}
	m_+(k) = - \overline{m_-(k)},  \qquad k\in\bR_+,
\end{equation}
which also holds for all classical reflectionless potentials. In addition, paper~\cite{Mar91} discusses the corresponding generalized soliton solutions of the KdV equation.

Later, Hur, McBride, and Remling~\cite{Remling} took that property of the related $m$-functions as defining the notion of reflectionless Schr\"odinger (or Jacobi) operators. In the Schr\"odinger operator context, they called a locally integrable potential~$q$ (or, more precisely, the Schr\"odinger operator~$T_q$) \emph{reflectionless} if the corresponding Weyl--Titchmarsh functions $m_\pm$ satisfy~\eqref{eq:int.m-refl} a.e.\ on $\bR_+$. 
The authors then characterised such potentials in terms of some representing measure $\sigma$ for the Herglotz functions~$m_\pm$ and derived some other their properties. 

A drawback of the approach of~\cite{Mar91,Remling} is that spectral properties of the Schr\"odinger operators with potentials in~$\overline{B(-\mu^2)}$ are not easy to get; they are only implicitly encoded in the measure~$\sigma$ related to the $m$-functions $m_\pm$. On the contrary, the approach of Gesztesy a.o.~\cite{GKZ} allowed to construct operators and soliton solutions to KdV with prescribed properties. Namely, the authors suggested to pass to the limit as $n\to\infty$ in formula~\eqref{eq:int.q-refl} for classical reflectionless potentials and in formula~\eqref{eq:int.soliton} for the $n$-soliton solutions of the KdV equation. They gave some conditions on sequences $\varkappa_j$ and $m_j$ under which such a passage to the limit is possible; the most crucial condition was that 
\[
	\sum_{j\ge 1}\frac{m_j^2}{\varkappa_j} < \infty,
\]
which guarantees that the determinant of the corresponding infinite matrix exists. It was proved that the limits~$q$ of~\eqref{eq:int.q-refl} produce reflectionless Schr\"odinger operators and the limits of~\eqref{eq:int.soliton} are classical solutions of the KdV equation. One of the most striking results of~\cite{GKZ} is that the sequence of $\varkappa_j$ can be fairly arbitrary, and thus the Schr\"odinger operators constructed that way may have an arbitrary countable (bounded below) set of negative bound states~$-\varkappa_j^2$ and, therefore, an arbitrary bounded below negative spectrum; the corresponding limit potentials are nevertheless bounded and smooth and generate classical solutions of the KdV equation.  In the case where the sequence $(\varkappa_j)_{j\in\bN}$ is in addition summable, the obtained generalized reflectionless potentials~$q$ and corresponding Schr\"odinger operators~$T_q$ allowed a more explicit description. For instance, the potential~$q$ is then integrable on the whole line, and the numbers~$m_j$ continue to be norming constants for the eigenvalues~$-\varkappa_j^2$.

The main aim of this paper is to specify the results of Marchenko~\cite{Mar91} and Gesztesy a.o.~\cite{GKZ} in the class of integrable reflectionless potentials. In that case, the negative spectrum of~$T_q$ is discrete and consists of at most countably many negative eigenvalues~$-\varkappa_1^2 < - \varkappa_2^2 < \dots <-\varkappa_N^2$, with $N\le \infty$. By the Lieb--Thirring inequality~\cite{LT,HLT,Weidl}, the sequence $(\varkappa_n)_{n\in\bN}$ is then summable. Moreover, for each eigenvalue~$-\varkappa_n^2$ the corresponding (right) Jost solution $e(\,\cdot\,, i\varkappa_n;q)$ is square integrable and thus one can introduce the norming constant~$m_n$. It turns out (and that is probably the most surprising fact, see Corollary~\ref{cor:one-to-one}) that there is absolutely no restrictions on the norming constants save that $m_n >0$. We prove that fact by exploiting an alternative formula for the Bargmann potentials, not using determinants. That alternative formula allowed us to pass to the limit as $n\to\infty$ in the classical formula and to derive the explicit formula~\eqref{eq:form.Q} and~\eqref{eq:form.q} for such generalised reflectionless potentials, thus giving their complete characterisation. 

We also prove that every such reflectionless potential is uniquely determined by the scattering data, the sequences of negative eigenvalues and the corresponding norming constants. To that end, we used the characterisation of generalized reflectionless potentials~$q$ due to Marchenko~\cite{Mar91} and Hur~a.o.~\cite{Remling} and showed that such a $q$ is uniquely determined by three negative spectra, that of $T_q$ and its half-line restrictions $T_q^+$ and $T_q^-$ by the Dirichlet condition~$y(0)= 0$. 
As a by-product, we also described all possible discrete spectra of $T_q$, $T_q^+$ and $T_q^-$ for generic reflectionless integrable $q$. We note that on that way we had to derive a generalisation a known formula relating the norming constant and the three discrete spectra and to prove some interpolation result for related Blaschke products, which can be of independent interest. On the other hand, we actually show that each integrable reflectionless potential is a limit in~$L_1(\bR)$ of a sequence of classical reflectionless potentials with special spectral properties, which opens the straightforward way to use the continuity results of our previous work~\cite{HM-traces} and generalise many classical relations (e.g.\ for Jost functions, $m$-functions etc) to this wider class.

Yet another advantage of describing reflectionless potentials by formulae~\eqref{eq:form.Q}--\eqref{eq:form.q} is that their straightforward modification produces solutions of the KdV equation, which can be called generalised soliton solutions. As we shall see, the proof is self-contained and is mainly based on the special algebraic structure of the solutions (combined with approximation arguments as necessary).

We conclude this introduction by describing how the paper is structured. In Section~\ref{sec:prelim}, we combine some preliminaries on the Jost solutions, scattering coefficients, their continuity on the potential, and properties of the classical reflectionless potentials to be used throughout the paper. Section~\ref{sec:existence} starts with providing insight for formulae~\eqref{eq:form.Q}--\eqref{eq:form.q}, which are proved in Theorem~\ref{thm:form.refl} to give integrable reflectionless potentials with prescribed eigenvalues and norming constants. To justify uniqueness theorem~\ref{thm:uniq-sp}, we first prove in Section~\ref{sec:uniqueness} that each integrable reflectionless potential~$q$ is uniquely determined by the discrete spectra of $T_q$ and its Dirichlet half-line restrictions $T_q^+$ and $T_q^-$, and then combine it with the three-spectra formula~\eqref{eq:3-sp} and an interpolation result for Blaschke products (Theorem~\ref{thm:interpolation}) proved in Section~\ref{sec:uniq-scattering}. Finally, in the last Section~\ref{sec:KdV} we prove that natural modification of~\eqref{eq:form.Q}--\eqref{eq:form.q} with time-evolving norming constants~$m_j$ produces a classical solution to the KdV equation. The appendix contains some auxiliary results on special class of Herglotz functions, Blaschke products, and relations between self-adjoint operators and their inverses.

\section{Preliminaries}\label{sec:prelim}

Throughout the paper, we denote by $\cQ_1$ the space of all real-valued functions in $L_1(\bR)$ with the inherited topology of the latter; $\bC_+$ and $\bC_-$ denote the open upper and lower half-planes, respectively. Every statement involving the $\pm$ signs should be regarded as two separate statements, with $\pm$ replaced with $+$ in the first statement and with $-$ in the second.

In this section, we collect some facts from~\cite{HM-traces, Mar} that will essentially be used in the proof of the main results of this paper. 

\subsection{Jost solutions and scattering coefficients}
For each $q \in \cQ_1$ and $\lambda \in \overline{\bC_+}\setminus\{0\}$, the equation 
\begin{equation}\label{eq:pre.equat}
	-y'' + qy = \lambda^2 y
\end{equation}
has a unique solution $e_+(\cdot,\lambda;q)$ that is asymptotic to $e^{i\lambda x}$ at $+\infty$, i.e., such that 
\[
	\lim_{x \to +\infty} e^{-i \lambda x} e_+(x,\lambda;q) = 1.
\]
Such a solution is called the \emph{(right) Jost solution}. Note that up to a constant factor, the Jost solution $e_+(\cdot,\lambda;q)$ with $\lambda \in \bC_+$ is the only solution of equation~\eqref{eq:pre.equat} that is square integrable at~$+\infty$. 

Analogously, for every $\lambda \in \overline{\bC_-}\setminus\{0\}$ one introduces the \emph{(left) Jost solution} $e_-(\cdot,\lambda;q)$ satisfying the relation 
\[
	\lim_{x \to -\infty} e^{-i \lambda x} e_-(x,\lambda;q) = 1;
\]
for $\lambda \in \bC_-$, it is the only solution of~\eqref{eq:pre.equat} (up to multiplication by a constant) that is square integrable at $-\infty$.

\begin{lemma}[\!\!\cite{HM-traces}]\label{lem:pre.Jost-bounds}
	Assume that $q \in \cQ_1$; then the right Jost solution satisfies the following inequality for all $x\in\bR$ and $\lambda\in \overline{\bC_+}\setminus\{0\}$:
\[
	|e_+(x,\lambda;q) - e^{i \lambda x}| 
		\le \frac{\|q\|_1}{|\lambda|} \exp\{ \|q\|_1/|\lambda| - |\opn{Im} \lambda|x\}.
\]	
	In addition, the Jost solution depends continuously on~$q$ in the following sense: if 	
	$\tilde{q}$ is another potential in $\cQ_1$ and $\alpha:= \max\{\|q\|_1,\|\tilde{q}\|_1\}$, then 
\[
	|e_+(x,\lambda;q) - e_+(x,\lambda;\tilde{q})| 
	\le \frac{\|q-\tilde{q}\|_1}{|\lambda|} \exp\{ 2\alpha/|\lambda| - |\opn{Im} \lambda|x\}
\]	
for all $x\in\bR$ and all $\lambda \in \overline{\bC_+}\setminus\{0\}$.	
\end{lemma}

\begin{corollary}\label{cor:pre.Jost-unif-int}
	For every $\delta>0$, we have 
	\begin{equation}\label{eq:pre.Jost-unif}
		\int_{\pm x>K} |e_\pm(x,\lambda;q)|^2\,dx =o(1), \qquad K\to +\infty,
	\end{equation}
	uniformly in $\lambda$ and $q$ satisfying $\pm\myIm\lambda >\delta$ and $\|q\|_1\le 1/\delta$, respectively.
\end{corollary}	
	
For a potential $q\in \cQ_1$ and every real non-zero~$k$, the left Jost solutions $e_-(\,\cdot\,,k;q)$ and $e_-(\,\cdot\,,-k;q)$ form a fundamental system 
of solutions to the energy equation
\[
	-y'' + q(x) y = k^2 y;
\]
therefore, there are unique coefficients $a(k)$ and $b(k)$ such that 
\begin{equation}\label{eq:pre.ab}
	e_+(x,k;q) = a(k) e_-(x,k;q) + b(k) e_-(x,-k;q).
\end{equation}
Using the asymptotic properties of the Jost solutions, we conclude that the function~$a$ satisfies the relation
\[
	2ik a(k) = W\bigl(e_+(\cdot,k;q),e_-(\cdot, - k;q)\bigr),
\]
with $W(f,g) = f'g-fg'$ being the Wronskian of functions $f$ and $g$. In particular, the above formula can be used to extend $a$ to the open upper-half complex plane~$\bC_+$; the extended function (still denoted by $a$) is analytic in $\bC_+$ and continuous up to the punctured real line $\bR\setminus\{0\}$.

Applying the same arguments as in the classical case of Faddeev--Marchenko potentials, i.e., real-valued potentials $q$ from $L_1(\bR; (1+|x|)dx)$ (cf.~\cite{Mar}), we conclude that $|a(k)|^2 = 1 + |b(k)|^2$, so that $a(k)$ never vanishes on~$\bR\setminus\{0\}$; the functions 
\[
	t(k):= \frac1{a(k)},  \qquad r_-(k):= \frac{b(k)}{a(k)}
\]
are called the \emph{transmission} and (left) \emph{reflection} coefficients. The (right) reflection coefficient $r_+$ can be derived analogously; as in the calssical case of Faddeev--Marchenko potentials,
\[
	r_+(k) = - \frac{b(-k)}{a(k)}.
\]

\begin{definition}
A real-valued function~$q$ in $L_1(\bR)$ is called a (generalized) \emph{reflectionless} potential if the corresponding reflection coefficients vanish identically on $\bR\setminus\{0\}$. The set of all (generalized) reflectionless potentials in $L_1(\bR)$ will be denoted by $\mathcal{R}_1$. 
\end{definition}

Clearly, a potential $q \in \cQ_1$ is reflectionless if and only if the transmission coefficient satisfies the condition~$|t(k)|\equiv 1$, or if $|a(k)|\equiv 1$ for $k\in\bR\setminus\{0\}$. In Subsection~\ref{ssec:3sp-reflectionless}, we mention an alternative (and more general) definition of reflectionless potentials in terms of the associated Weyl--Titchmarsh $m$-functions; the corresponding relation was first derived by Marchenko~\cite{Mar} and then turned into definition by Hur~a.o.~\cite{Remling}; see also~\cite{Rem13, PolRem09} for similar treatments of reflectionless Jacobi matrices.

\subsection{Continuity of eigenvalues and norming constants}

For each $q\in\cQ_1$ with $n\le\infty$ negative eigenvalues $-\varkappa_1^2 < -\varkappa_2^2 < \dots$, we define a non-increasing sequence $\bm{\varkappa}(q):=(\varkappa_j)_{j=1}^\infty$ of non-negative numbers, with $\varkappa_{n+1} = \varkappa_{n+2} = \dots =0$ if $n$ is finite. By the Lieb--Thirring inequality~\cite{LT} (for the case under consideration proved by Weidl~\cite{Weidl} and with the exact constant established by Hundertmark a.o.~\cite{HLT}), we have 
\[
	\|\bm{\varkappa}(q)\|_1 =\sum_{j\ge1} \varkappa_j \le \tfrac12\|q\|_1.
\] 
In particular, the sequence~$\bm{\varkappa}(q)$ belongs to $\ell_1(\bN)$;  one of the main results of~\cite{HM-traces} states that, moreover, the mapping 
\[
	\cQ_1 \ni q \mapsto \bm{\varkappa}(q) \in \ell_1(\bN)
\]
is continuous. 

Assume that $-\varkappa^2$ is an eigenvalue of a Schr\"odinger operator $T_q$ with $q\in \cQ_1$. Then the corresponding eigenfunction $\psi$ must be collinear to the right Jost solution $e_+(\cdot, i\varkappa;q)$ at $+\infty$ and to the left Jost solution $e_-(\cdot, -i\varkappa;q)$ at $-\infty$; therefore, the two Jost solutions are collinear, i.e., 
\[
	e_+(\cdot,i\varkappa;q) = C_+ e_-(\cdot,-i\varkappa;q)
\]
with $C_+ = e_+(x,i\varkappa;q)/e_-(x,-i\varkappa;q)$. In particular, the right Jost solution is an eigenfunction, and the (right) \emph{norming constant} $m_+$ corresponding to the eigenvalue $-\varkappa^2$ is the inverse $L_2$-norm of this Jost solution, i.e.,
\[
	m_+^{-2} := \int_{\bR} |e_+(x,i\varkappa;q)|^2\,dx.
\]

One can similarly introduce the left norming constant, viz.
\[
	m_-^{-2} := \int_{\bR} |e_-(x,-i\varkappa;q)|^2\,dx.
\]
Clearly, both $m_+e_+(\cdot,i\varkappa;q)$ and $m_-e_-(\cdot,-i\varkappa;q)$ are eigenfunctions of the operator $T_q$ for the eigenvalue $-\varkappa^2$ of norm one; therefore, they differ by a unimodular factor (in fact, $\pm1$); in particular,
\begin{equation}\label{eq:pre.normed-ef}
	m_+|e_+(\cdot,i\varkappa;q)| \equiv m_-|e_-(\cdot,-i\varkappa;q)|. 
\end{equation}
The right and left norming constants are closely related; namely, arguments similar to those in~\cite[Ch.~3.5]{Mar} show that 
\begin{equation}\label{eq:pre.norming-relation}
	(m_+m_-)^{-2} = - \bigl(\dot{a}(i\varkappa)\bigr)^2.
\end{equation}

In what follows, we shall mostly consider the right norming constants and thus will be writing $m$ instead of $m_+$ whenever no confusion can arise.

Next we show certain continuity of the norming constants. Assume that $q_n$ is a sequence of integrable potentials with the property that $-\varkappa^2$ is an eigenvalue of every operator $T_{q_n}$ and let $m_n$ be the corresponding norming constant. Assume next that, as $n\to\infty$, the sequence $q_n$ converges in $L_1(\bR)$ to a function~$q_0$. Continuity of the eigenvalues~\cite{HM-traces} guarantees that $-\varkappa^2$ is an eigenvalue of the operator~$T_{q_0}$; let $m_0$ be the corresponding norming constant.

\begin{lemma}\label{lem:pre.cont-norm}
Under the above assumptions, the limit of $m_n$ as $n\to\infty$ exists and is equal to $m_0$.
\end{lemma}

\begin{proof} We prove that the eigenfunctions $e_+(\cdot, i\varkappa;q_n)$ converge in the topology of $L_2(\bR)$ to the eigenfunction~$e_+(\cdot, i\varkappa;q_0)$.
	By Lemma~\ref{lem:pre.Jost-bounds}, there exists a $c>0$ independent of~$n$ such that  
	\[
		|e_+(x, i\varkappa;q_n) - e_+(x, i\varkappa;q_0)|^2 \le c \|q_n - q_0\|_1^2 e^{-2 \varkappa x}
	\]
	and thus, for every $a\in\bR$,
	\[
		\int_{a}^\infty |e_+(x, i\varkappa;q_n) - e_+(x, i\varkappa;q_0)|^2\,dx \to 0
	\]
	as $n\to\infty$. By similar arguments applied to the left Jost solutions, 
	\[
		\int_{-\infty}^a |e_-(x, -i\varkappa;q_n) - e_-(x, -i\varkappa;q_0)|^2\,dx \to 0
	\]
	as $n\to\infty$ for every $a\in\bR$. Choose now $a\in\bR$ such that $e_-(a,-i\varkappa;q_0)\ne0$; then pointwise convergence of the Jost solutions guarantees that also $e_-(a,-i\varkappa;q_n)\ne0$ for all $n$ large enough and hence that 
	\[
		C_n := \frac{e_+(a,i\varkappa;q_n)}{e_-(a,-i\varkappa;q_n)} \to \frac{e_+(a,i\varkappa;q_0)}{e_-(a,-i\varkappa;q_0)} =: C_0
	\]
	as $n\to\infty$. Now the estimate
	\begin{align*}
			\int_{-\infty}^a |e_+(x, i\varkappa;q_n) - e_+(x, i\varkappa;q_0)|^2\,dx  
				& \le 2C_n^2 \int_{-\infty}^a |e_-(x, -i\varkappa;q_n) - e_-(x, -i\varkappa;q_0)|^2\,dx\\
				& + 2(C_n-C_0)^2 \int_{-\infty}^a |e_-(x, -i\varkappa;q_0)|^2\,dx
	\end{align*}
	 along with the above convergence results for $C_n$ and $e_-(\cdot, -i\varkappa;q_n)$ completes the proof. 
\end{proof}

\subsection{Properties of the classical reflectionless potentials}\label{ssec:classical-properties}
We conclude this section by listing some properties of the classical and generalized reflection potentials; cf.~\cite{Mar,GKZ,HM-traces,Remling} and the references therein. 

Assume that $q$ is a reflectionless potential such that the corresponding Schr\" odinger operator possesses $n$ negative eigenvalues $-\varkappa_1^2 < \dots < -\varkappa_n^2$ and (right) norming constants $m_1, \dots, m_n$. Then 
\begin{itemize}
	\item[(a)] $q$ is negative and is equal to
	\[
	q(x) = -4 \sum_{j=1}^n \varkappa_j m_j^2|e_+(x,i\varkappa_n;q)|^2;
	\] 
	in particular, the following trace formula holds:
	\[
	\int_\bR |q(x)|\,dx = 4 \sum_{j=1}^n \varkappa_j;
	\]
	\item[(b)] $q$ admits an analytic continuation in the strip 
	\[
		 \Pi_{\varkappa_1} := \{ z = x + iy \in \mathbb{C} \mid |y| < \varkappa_1^{-1}\}
	\]
	and obeys therein the inequality 
	\begin{equation}\label{eq:pre.q-analytic}
		|q(x + iy)| \le 2\varkappa_1^2 (1- \varkappa_1|y|)^{-2};
	\end{equation}
	\item[(c)] the corresponding scattering coefficient $a$ is a rational function given by 
	\[
		a(z) = \prod_{j=1}^{n} \frac{z-i\varkappa_j}{z+i\varkappa_j}.
	\]	
\end{itemize}		
	
We observe that~\eqref{eq:pre.normed-ef} also implies an alternative representation of $q$ with the left Jost solutions and the left norming constants $m_{1,-}, \dots m_{n,-}$, viz.
\[
	q(x) = -4 \sum_{j=1}^n \varkappa_j m_{j,-}^2|e_-(x,-i\varkappa_j;q)|^2.
\]

If $q\in\cQ_1$ is reflectionless, then by the results of~\cite{HM-traces}, $a$ is given by a similar Blaschke product (finite or inifinite depending on the number of negative eigenvalues), i.e.,
\[
	a(z) = \prod_{j\ge1} \frac{z-i\varkappa_j}{z+i\varkappa_j}.
\]
Such a function~$a$ is unimodular on $\bR$, so that indeed $r\equiv 0$.
Also, the function $a(\cdot;q)$ was proved in~\cite{HM-traces} to depend continuously on~$q\in\cQ_1$ in the topology of uniform convergence on compact subsets of $\bR\setminus\{0\}$; in particular, this implies that the family of all reflectionless potentials that are integrable on the whole line form a closed subset of $L_1(\bR)$.

Denote by $B(-\varkappa^2)$ the set of all classical reflectionless potentials~$q$ for which the corresponding Schr\"odinger operators $S_q$ have spectra contained in $[-\varkappa^2,\infty)$. Marchenko~\cite{Mar91} studied the closure $\overline{B(-\varkappa^2)}$ of $B(-\varkappa^2)$ in the topology of uniform convergence on compact subsets of~$\bR$; potentials in~$\overline{B(-\varkappa^2)}$ admit analytic extension in the strip $\Pi_{\varkappa_1}$ and obey therein the bound~\eqref{eq:pre.q-analytic}.  One of the important results of~\cite{Mar91} (proved earlier in~\cite{Lun}) reads

\begin{proposition}\label{pro:pre.Marchenko}
The set $\overline{B(-\varkappa^2)}$ is compact with respect to the uniform convergence on compact subsets of~$\bR$. 
\end{proposition}

We build upon this result and construct sequences of classical reflectionless potentials that, in addition, converge in the topology of~$L_1(\bR)$. 

\begin{definition}\label{def:pre.condA} Assume that $\bm{\varkappa} =(\varkappa_j)_{j\in\bN}$ is an arbitrary strictly decreasing sequence of positive numbers belonging to~$\ell_1(\bN)$. We say that a sequence~$(q_n)_{n\in\bN}$ of classical reflectionless potentials satisfies assumption $A(\bm{\varkappa})$ if the following two conditions are satisfied:
	\begin{enumerate}
		\item for every $n\in\bN$, the set $\{-\varkappa^2_j\}_{j=1}^n$ is the negative spectrum of the operator $T_{q_n}$;
		\item denote by $m_{j,n}$, $j\le n$, the right norming constants of the operator $T_{q_n}$ corresponding to the eigenvalues $-\varkappa^2_j$; then, for every $j\in\bN$,  the limit of $m_{j,n}$ as $n\to\infty$ exists and is positive.
	\end{enumerate}
\end{definition}

\begin{lemma}\label{lem:pre.condA}
	Assume that $(q_n)_{n\in\bN}$ is a sequence of classical reflectionless potentials satisfying assumption $A(\bm{\varkappa})$. Then there exists a subsequence of  $(q_n)_{n\in\bN}$ that converges uniformly on compact subsets of $\mathbb{R}$ as  well as in the norm of the space $L_1(\mathbb{R})$ to some reflectionless potential $q\in\cQ_1$; moreover, $\|q\|_1=4\sum_{j=1}^\infty \varkappa_j$.
\end{lemma}

\begin{proof} By Proposition~\ref{pro:pre.Marchenko}, there is a subsequence~$(q_{n_k})_{k\in\bN}$ of $(q_n)_{n\in\bN}$ converging uniformly on compact subsets of~$\bR$. We prove next that the sequence $(q_n)_{n\in\bN}$ is uniformly integrable, i.e., that 
	\[
		\int_{|x|>K} |q_n(x)| \,dx =o(1), \qquad K\to +\infty,
	\]
	uniformly in $n\in\bN$. This implies that the subsequence $(q_{n_k})_{k\in\bN}$ converges also in the $L_1(\bR)$-topology; the limit $q\in L_1(\bR)$ is then reflectionless and
	\[
		\|q\|_1=\lim_{k\to\infty}\|q_{n_k}\|_1=4\sum_{j=1}^\infty \varkappa_j
	\]
	as claimed.

	Recall that $q_n$ is equal to
	\[
	q_n(x) = -4 \sum_{j=1}^n \varkappa_n m_{j,n}^2|e_+(x,i\varkappa_j;q_n)|^2
	\]
	and that its $L_1$-norm is equal to $4\sum_{j=1}^n \varkappa_j$; in particular, $\|q_n\|_1\le \alpha:= 4 \sum_{j\ge1} \varkappa_j$. Given $\eps>0$, we first find $M\in\bN$ such that
	\[
	\sum_{j>M} \varkappa_j < \eps/8
	\]
 	and then set
	\[
		\widetilde{q_n}(x) = -4 \sum_{1\le j\le M} \varkappa_j m_{j,n}^2|e_+(x,i\varkappa_j;q_n)|^2.
	\]  
	By the definition of the norming constants $m_{j,n}$ we see that, for every $n\in\bN$,
	\[
		\|q_n- \widetilde{q_n}\|_1\le 4\sum_{j>M} \varkappa_j < \eps/2;
	\]   
	also, the assumptions of the lemma imply that 
	\[     
		\sup_{n\ge M} \sum_{1\le j\le M} \varkappa_j m_{j,n}^2=:C<\infty.
	\]
	By Corollary~\ref{cor:pre.Jost-unif-int}, there exists $K>0$ such that
	\[
	\int_{x>K} |e_+(x,i\varkappa_j;q_n)|^2 \,dx < \eps/8C
	\]
	for all $j\le M$ and all $n\in\bN$. As a result, we conclude that
	\[ 
		\int^{+\infty}_K |q_n(x)| \,dx
			\le \eps/2+ \int^{+\infty}_K |\widetilde{q_n}(x)| \,dx 	
			\le \eps/2 +\eps/2=\eps,
	\]
	yielding the required result on the positive half-line, i.e., that 
	\begin{equation}\label{eq:pre.unif-int}
		\int_{x>K} |q_n(x)| \,dx =o(1), \qquad K\to +\infty
	\end{equation}
	uniformly in $n\in\bN$. 

	We next prove an analogous result on the negative half-line. Firstly, observe that the formula
	$$   
		\widehat{q}_n(x): = q_n(-x), \qquad x\in\mathbb{R}, \quad n\in\bN,
	$$
	defines a sequence $(\widehat{q}_n)_{n\in\bN}$ of reflectionless potentials such that the negative spectrum of the operators $T_{\hat{q}_n}$ coincides with the set $\{-\varkappa^2_j\}_{j=1}^n$. In addition, the right norming constant  $m_j(T_{\hat{q}_n})$ of $T_{\hat{q}_n}$ is equal to $m_{j,n,-}$, the left norming constant of the original operator $T_{q_n}$ corresponding to the eigenvalue $-\varkappa_j^2$.

	In view of~\eqref{eq:pre.norming-relation}, we find that
	\[
		m_{j,n,-}^{2} = - m_{j,n}^{-2} \bigl(\dot{a}_n(i\varkappa_j)\bigr)^{-2},
	\]
	with
	\[
		{a_n}(z):= \prod_{j=1}^n \frac{z-i\varkappa_j}{z+i\varkappa_j}.
	\]
	It follows from Lemma~\ref{lem:pre.Blaschke} that
	\[
		\lim\limits_{n\to \infty}\bigl(\dot{a}_n(i\varkappa_j)\bigr)^{-1} = \bigl(\dot{a}(i\varkappa_j)\bigr)^{-1},
	\]
	where
	\[
		{a}(z):= \prod_{j=1}^\infty \frac{z-i\varkappa_j}{z+i\varkappa_j}.
	\]
	Therefore, the sequence $(\widehat{q}_n)_{n\in\bN}$ of the classical reflectionless potentials satisfies the assumption $A(\bm{\varkappa})$. 
	By virtue of the estimate~\eqref{eq:pre.unif-int} established above, we conclude that 
	$$
		\int^{+\infty}_K |q_n(-x)|^2\,dx =o(1), \qquad K\to +\infty,
	$$
	uniformly in all $n\in\bN$. The proof is complete.
\end{proof}


\section{Reflectionless potentials with prescribed spectral data: existence}\label{sec:existence}


\subsection{Classical reflectionless potentials, revisited}

In this subsection, we rewrite formula~\eqref{eq:int.q-refl} for the classical reflectionless potentials in a way that will allow direct generalizations to the case~$q\in\cQ_1$.

Assume that $q$ is a real-valued potential of Faddeev--Marchenko class, i.e., that 
\[
	\int_{\bR} (1 + |x|)|q(x)|\,dx < \infty. 
\]
According to the classical inverse scattering theory for Schr\"odinger operators on the line~\cite{Mar}, the potential~$q$ satisfies the relation 
\[
	q(x) = - 2 \frac{d}{dx} k(x,x),
\]
where $k(x,t)$ is the kernel of the so-called transformation operator. 
This kernel $k$ can be obtained as a solution of the Marchenko equation
\begin{equation}\label{eq:form.M-equation}
	k(x,t) + f(x+t) + \int_x^\infty k(x,s) f(s+t)\,ds =0, \qquad x < t,
\end{equation}
in which $f$ encodes the scattering data for $T_q$, i.e., the reflection coefficient~$r_+$, the negative spectrum $-\varkappa_1^2 < -\varkappa_2^2 < \dots < -\varkappa_n^2<0$ and the corresponding norming constants $m_1, m_2, \dots, m_n$: 
\[
	f(s) := \sum_{j=1}^n m^2_j e^{-\varkappa_j s} + \frac1{2\pi} \int_{\bR} r_+(k) e^{iks}\,dk.
\]

In the reflectionless case ($r_+\equiv0$), equation~\eqref{eq:form.M-equation} is degenerate of rank~$n$, and thus can be solved explicitly. In that case the solution
$k$ must be of the form
\begin{equation}\label{eq:form.k-sum}
k(x,t) = \sum_{j=1}^n g_j(x) m_j e^{-\varkappa_j t};
\end{equation}
plugging that expression into the Marchenko equation, we arrive at the following linear system of equations for determining $g_j$:
\[
g_j(x) +  m_je^{-\varkappa_jx} +
\sum_{l=1}^n g_l(x) \int_x^\infty m_lm_j e^{-(\varkappa_l+\varkappa_j)s}\,ds = 0, \qquad j = 1,\dots, n,
\]
or 
\begin{equation}\label{eq:form.k-system}
g_j(x) +  m_je^{-\varkappa_jx} +
\sum_{l=1}^n g_l(x) \frac{m_lm_j}{\varkappa_l + \varkappa_j} e^{-(\varkappa_l+\varkappa_j)x} = 0, \qquad j = 1,\dots, n.
\end{equation}
Solving this via Cramer's rule and then plugging the result into~\eqref{eq:form.k-sum} gives the Kay--Moses formula~\eqref{eq:int.q-refl}.

Gesztesy a.o.~\cite{GKZ} proved that, under some conditions on $\varkappa_j$ and $m_j$, one can pass to the limit in the Kay--Moses formula to get a generalized reflectionless potential. One of the most essential conditions in~\cite{GKZ} was that 
\[
\sum_{j=1}^\infty \frac{m^2_j}{\varkappa_j} <\infty,
\]
imposing a strong restriction on norming constants and thus not allowing complete characterization of all integrable reflectionless potentials and their scattering data. 

This is the reason we decided to use a slightly different approach leading to a formula for all reflectionless potentials in~$\cQ_1$. 
Namely, in the Euclidean space~$\bC^n$ we introduce the column vector
\[
\Phi_n(x) := \bigl(m_1e^{-\varkappa_1x},\dots,m_ne^{-\varkappa_nx}\bigr)^\top
\]
and $G(x):= \bigl(g_1(x),\dots,g_n(x)\bigr)^\top$; then $k(x,t) = G^\top(x)\Phi(t)$, and the above system reduces to the vector-valued equation
\[
G(x) +  \int_x^\infty \Phi(s) \Phi^\top(s)\,ds \,G(x) = - \Phi(x).
\]
Since the matrix
\[
M(x):= \int_x^\infty \Phi(s) \Phi^\top(s)\,ds
\]
is nonnegative, $I+M(x)$ is nonsingular, and the above equation has a unique solution
\[
G(x) =  - (I+M(x))^{-1}\Phi(x);
\]
the kernel $k$ of the transformation operator then is
\[
k(x,t) = G^\top(x) \Phi(s) = - \Phi^\top(x) (I + M(x))^{-1} \Phi(t),
\]
and the potential $q$ is given by the formula
\[
q(x) = -2 \frac{d}{dx}k(x,x) = 2 \frac{d}{dx} \Phi^\top(x) (I + M(x))^{-1} \Phi(x).
\]

We next observe that the $(k,l)$-entry of the matrix $M(x)$ is equal to 
\[
\bigl(M(x)\bigr)_{k,l} = \frac{m_km_l}{\varkappa_k + \varkappa_l}e^{-(\varkappa_k+\varkappa_l)x}
\]
To make possible passage to the limit as $n\to\infty$, we introduce the vector $\bm{\varkappa}_n := (\varkappa_1, \dots, \varkappa_n)^\top$ in $\bC^n$, 
the $n\times n$ matrix $\Gamma_n$ with entries 
\[
(\Gamma_n)_{k,l}:= \frac{\varkappa_k\varkappa_l}{\varkappa_k + \varkappa_l},
\]
and two diagonal matrices $A_n$ and $K_n$,
\[
A_n = \diag\{\alpha_1, \dots,\alpha_n\}, \qquad K_n = \diag\{\varkappa_1,\dots, \varkappa_n\}
\]
with $\alpha_j:=\varkappa_j/m_j$. 
With these notations, we see that $\Phi_n(x) = A_n^{-1}e^{-K_nx} \bm{\varkappa}_n$, 
\[
M(x) = A_n^{-1}e^{-K_nx} \Gamma_n A_n^{-1}e^{-K_nx},
\]
and 
\begin{multline}\label{eq:form.Q-classical}
	Q_n(x) := \Phi^\top(x) (I + M(x))^{-1} \Phi(x) \\
	= \bm{\varkappa}_n^\top (A_n^2 e^{2K_nx} + \Gamma_n)^{-1} \bm{\varkappa}_n
= \langle (A_n^2 e^{2K_nx} + \Gamma_n)^{-1} \bm{\varkappa}_n, \bm{\varkappa}_n \rangle_{\bC^n},
\end{multline}
where $\langle \cdot\,,\,\cdot \rangle_{\bC^n}$ is the standard scalar product in~$\bC^n$. Combining the above relations (and writing $q_n$ instead of $q$ for consistency), we conclude that 
\begin{equation}\label{eq:form.q-classical}
	q_n(x) = 2 Q'_n(x)  = 2 \frac{d}{dx} \langle (A_n^2 e^{2K_nx} + \Gamma_n)^{-1} \bm{\varkappa}_n, \bm{\varkappa}_n \rangle_{\bC^n}.
\end{equation}

\subsection{Formula for integrable reflectionless potentials}\label{ssec:formula}

The above formula allows a direct infinite-dimensional generalization. Namely, take an arbitrary positive and strictly decreasing sequence $\bm{\varkappa}=(\varkappa_n)_{n\in\bN}$ in $\ell_1(\bN)$ and an arbitrary sequence $\mathbf{m}=(m_n)_{n\in\bN}$ of positive numbers and define an auxiliary sequence $\bm{\alpha} = (\alpha_1,\alpha_2,\dots)$ with $\alpha_j := \varkappa_j/m_j$. Next, in the Hilbert space~$H:=\ell_2(\bN)$ with the scalar product $\langle \cdot, \cdot \rangle$ and standard basis $\mathbf{e}_j$, we introduce positive diagonal operators 
\[
K = \diag\{\varkappa_1, \varkappa_2, \dots \}, \qquad A = \diag\{\alpha_1, \alpha_2, \dots\}
\]
and the operator $\Gamma$ defined via
\[
\langle \Gamma \mathbf{e}_k, \mathbf{e}_j\rangle= \frac{\varkappa_k\varkappa_l}{\varkappa_k + \varkappa_l} =: (\Gamma)_{k,l} .	
\]

\begin{theorem}\label{thm:form.refl} 
Under the above notations, the function
	\begin{equation}\label{eq:form.Q} 
	Q(x):= \|(A^2 e^{2xK} + \Gamma)^{-1/2} \bm{\varkappa}\|^2
	\end{equation}
	is well defined on the real line and the formula
	\begin{equation}\label{eq:form.q}
	q(x) := 2 Q'(x)
	\end{equation}
	defines a reflectionless potential in~$\cQ_1$ such that the negative spectrum of the Schr\"odinger operator $T_q$ coincides with the set $\{-\varkappa_n^2\}_{n\ge1}$ and the corresponding norming constant for $-\varkappa_n^2$ is $m_n$.
\end{theorem}

This theorem gives a complete description of the scattering data for the Schr\"odinger operators with potentials in $\cQ_1$: as we already know, the bound states generate summable sequences of~$\varkappa_j$, and there is absolutely no restriction on the norming constants $m_j$ except that $m_j>0$. 

In the rest of this subsection, we derive some auxiliary results which, in particular, will show that formula~\eqref{eq:form.Q} is well defined. We start with establishing some properties of the operator~$\Gamma$. 

\begin{lemma}\label{lem:form.gamma}
	The operator~$\Gamma$ is positive and of trace class.
\end{lemma}

\begin{proof}
	In the Hilbert space~$L_2(\bR_+)$, we introduce functions
	\[
	\phi_j(x):=e^{-\varkappa_j x},  \qquad j\in\bN,
	\]	
	and an operator $F: \ell_2(\bN) \to L_2(\bR_+)$ acting via 
	\[
	F \mathbf{c}:= \sum_{j\in\bN} \varkappa_j c_j \phi_j, \qquad \mathbf{c}=(c_j)_{j\in\bN}\in\ell_2(\bN).
	\]
	With $(\mathbf{e}_j)_{j\in\bN}$ being the standard orthonormal basis of $\ell_2(\bN)$, we have
	\[
	\sum_{j\in\bN} \|F \mathbf{e}_j\|^2 =  \sum\limits_{j=1}^\infty \frac{\varkappa_j^2}{2\varkappa_j} = \sum\limits_{j=1}^\infty \frac{\varkappa_j}{2} <\infty,
	\] 
	so that $F$ is a Hilbert--Schmidt operator.
	
	Assume that the kernel of $F$ is non-trivial. Then for some $n\in\bN$ we have
	\begin{equation}\label{eq.2B}
	\phi_n=\sum\limits_{j=n+1}^\infty  c_j \varkappa_{j} \phi_{j},
	\end{equation}
	with  $(c_j)_{j\in\bN}\in\ell_2$. As $\sum\limits_{j=n+1}^\infty  |c_j|\varkappa_{j}<\infty$, repeated differentiation of equality~\eqref{eq.2B} results in the relations  \begin{equation*}
	\phi_n=\sum\limits_{j=n+1}^\infty  \left(\frac{\varkappa_{j}}{\varkappa_{n}}\right)^s c_j \varkappa_j\phi_{j}
	\end{equation*}
	for all $s \in \bN$. Using the dominated convergence theorem and passing to the limit as $s\to\infty$, we conclude that~$\phi_n\equiv0$, which is a contradiction.
	
	By direct verification,	$\Gamma=F^*F$; as a result, the operator~$\Gamma$ is positive and of trace class.
\end{proof}

\begin{lemma}\label{lem:form.Gamma1/2}
	The vector $\bm{\varkappa}$ belongs to the domain of $\Gamma^{-1/2}$; moreover,
	\begin{equation}\label{eq:form.Gamma1/2}
	\lim_{t\to0+} \langle(tI+\Gamma)^{-1} \bm{\varkappa}, \bm{\varkappa}\rangle = \|\Gamma^{-1/2}\bm{\varkappa}\|^2 \le  2\sum_{j>1}\varkappa_j.
	\end{equation}
\end{lemma}

\begin{proof}
	Recall the notations $q_n$,$A_n$, $K_n$, $\Gamma_n$ and $\bm{\varkappa}_n$ introduced in the previous subsection. According to~\eqref{eq:form.q-classical}, we get 
	\[
		\frac12\int_{\bR} |q_n(t)|\,dt 
			= \lim_{x\to-\infty} \langle(A_n^2 e^{2K_nx}+\Gamma_n)^{-1} \bm{\varkappa}_n, \bm{\varkappa}_n\rangle_{\bC^n}
			= \langle\Gamma_n^{-1} \bm{\varkappa}_n, \bm{\varkappa}_n\rangle_{\bC^n} = 2 \sum_{j=1}^n \varkappa_j. 
	\]
	Denote now by $P_n$ the orthogonal projector in $H = \ell_2(\bN)$ onto the first $n$ coordinates and by $P_n':=I-P_n$ the complementing orthoprojector. Observe that the restriction of $P_n\Gamma P_n$ onto $P_n H$ is just $\Gamma_n$ and that $P_n\bm{\varkappa}$ can be identified with~$\bm{\varkappa}_n$. Take any $t>0$; then we get 
	\begin{align*}
		\langle (tI+P_n\Gamma P_n)^{-1} \bm{\varkappa}, \bm{\varkappa}\rangle 
			&= \langle (tP_n+P_n\Gamma P_n)^{-1} P_n\bm{\varkappa}, P_n\bm{\varkappa} \rangle + t^{-1}\|P_n'\bm{\varkappa}\|^2 \\
			&= \langle(tI_n+\Gamma_n)^{-1} \bm{\varkappa}_n, \bm{\varkappa}_n\rangle_{\bC^n} + t^{-1}\|P_n'\bm{\varkappa}\|^2,
	\end{align*}
	so that 
	\begin{align*}	
		\limsup_{n\to \infty} \langle (tI+P_n\Gamma P_n)^{-1} \bm{\varkappa}, \bm{\varkappa}\rangle 
	& = \limsup_{n\to \infty} \langle (tI_n+\Gamma_n)^{-1} \bm{\varkappa}_n, \bm{\varkappa}_n\rangle_{\bC^n} \\
	& \le \lim_{n\to \infty}\langle\Gamma_n^{-1} \bm{\varkappa}_n, \bm{\varkappa}_n\rangle_{\bC^n} = 2 \sum_{j=1}^\infty \varkappa_j. 
	\end{align*}
	On the other hand, since the operator $\Gamma$ is of trace class, its compressions $P_n\Gamma P_n$ converge in norm to~$\Gamma$. Therefore, the above limit superior exists in fact as the limit and 
	\[
		\lim_{n\to \infty} \langle (tI+P_n\Gamma P_n)^{-1} \bm{\varkappa}, \bm{\varkappa}\rangle 
			= \langle (tI+ \Gamma)^{-1} \bm{\varkappa}, \bm{\varkappa}\rangle 
			\le 2 \sum_{j=1}^\infty \varkappa_j.
	\]
	It follows from the latter inequality and Proposition~\ref{pro:app.domain12} that $\bm{\varkappa}$ belongs to the domain of $\Gamma^{-1/2}$ and that 
	\[
	\|\Gamma^{-1/2}\bm{\varkappa}\|^2 = \lim_{t\to0+}  \langle(tI+\Gamma)^{-1} \bm{\varkappa}, \bm{\varkappa}\rangle \le 2 \sum_{j=1}^\infty\varkappa_j.
	\]
	The proof is complete. 
\end{proof}

\begin{corollary}
	The function $Q$ of~\eqref{eq:form.Q} is well defined on the whole real line and satisfies there the bound $|Q(x)| \le 2 \sum_{j=1}^\infty\varkappa_j$. 
\end{corollary}

\begin{proof}
	For every $x\in\bR$, the operator $A^2 e^{2xK}$ is positive, so that by Proposition~\ref{pro:app.inverse} we get
	\[
	(A^2 e^{2xK} + \Gamma)^{-1} \le \Gamma^{-1}.
	\]
	Therefore, $\dom(\Gamma^{-1/2}) \subset \dom\bigl((A^2 e^{2xK} + \Gamma)^{-1/2}\bigr)$ and
	\[
	Q(x) = \|(A^2 e^{2xK} + \Gamma)^{-1/2} \bm{\varkappa}\|^2 \le \|\Gamma^{-1/2} \bm{\varkappa} \|^2 \le 2 \sum_{j=1}^\infty\varkappa_j
	\]
	as claimed.
\end{proof}

\subsection{Proof of Theorem~\ref{thm:form.refl}}

We give separate proofs of the theorem in two different cases, the first one when the operator~$A$ is uniformly positive and the second one for an arbitrary positive~$A$.
In both cases, we construct a sequence of classical reflectionless potentials that converges to $q$ in the $L_1(\bR)$-topology and guarantees that $T_q$ has the required negative spectrum and norming constants. The first case is much simpler and more straightforward, and that is the reason why we include it as well.

\medskip

\noindent\textbf{Case~1:} The operator $A$ is uniformly positive. For each $n\in\bN$, we denote by $q_n$ the classical reflectionless potential corresponding to the eigenvalues $-\varkappa_1^2< -\varkappa_2^2 < \dots < - \varkappa_n^2$ and the norming constants $m_1, m_2, \dots, m_n$. Then the sequence of potentials $(q_n)_{n\in\bN}$ satisfies the condition $A(\bm{\varkappa})$ of Definition~\ref{def:pre.condA} and thus there is a subsequence $(q_{n_k})_{k\in\bN}$ that converges uniformly on compact subsets of~$\bR$ as well as in the topology of~$L_1(\bR)$ to some reflectionless potential~$q_0$. 

We shall prove below that the corresponding sequence of bounded functions $Q_n$ of~\eqref{eq:form.Q-classical}, 
\[
	Q_n(x) = \langle (A_n^2 e^{2x K_n} + \Gamma_n)^{-1} \bm{\varkappa}_n, \bm{\varkappa}_n \rangle_{\bC^n}
\]
converges pointwise on~$\bR$ to the function~$Q$ of~\eqref{eq:form.Q}. As also, for every $x \in \bR$, 
\[
	Q_{n_k}(x) := - \frac12\int_x^\infty q_{n_k}(t)\,dt \to - \frac12\int_x^\infty q_{0}(t)\,dt =: Q_0(x),
\]
we conclude that $Q(x) = Q_0(x)$
and thus that 
\[
	q(x) = 2 Q'(x) = 2 Q_0'(x) = q_0(x).
\]
Convergence of $q_{n_k}$ to $q$ in $L_1(\bR)$ now implies that the operator $T_q$ has the required scattering data: it is reflectionless, its negative spectrum coincides with the eigenvalues $-\varkappa_n^2$, $n \in \bN$, and the corresponding norming constants are~$m_n$. 

Therefore, it remains to prove that $Q_n$ converge to $Q$ pointwise on~$\bR$. As the operator~$A^2e^{2xK}+ \Gamma$ is uniformly positive, its inverse is bounded and thus
\begin{equation}\label{eq:form.Q-positive-A}
	Q(x) = \langle (A^2 e^{2xK} + \Gamma)^{-1} \bm{\varkappa}, \bm{\varkappa} \rangle.
\end{equation}
Fix an arbitrary $x\in\bR$ and set for brevity $B:=A^2e^{2xK}$. Since $\|P_n\Gamma P_n - \Gamma\| \to 0$ as $n\to\infty$ and the operators $B + \Gamma$ and $B + P_n\Gamma P_n$ are uniformly positive and thus boundedly invertible, we conclude that the inverse operators converge in norm, i.e., that
\[
	\| (B + P_n\Gamma P_n)^{-1} - (B + \Gamma)^{-1} \| \to 0
\]
as $n\to\infty$. As a result, 
\begin{multline*}
	|Q_n(x) - Q(x)| 
		\le |\langle (B + P_n\Gamma P_n)^{-1} P_n\bm{\varkappa}, P_n\bm{\varkappa} \rangle - 
			\langle (B + \Gamma)^{-1}   P_n \bm{\varkappa},  P_n \bm{\varkappa} \rangle |  \\
		+  | \langle (B + \Gamma)^{-1}   P_n \bm{\varkappa},  P_n \bm{\varkappa} \rangle 
			- \langle (B + \Gamma)^{-1} \bm{\varkappa}, \bm{\varkappa} \rangle| \to 0,
\end{multline*}
and the proof is complete.

\medskip

\noindent\textbf{Case~2:} $A$ is strictly positive, but not uniformly positive. In that case the vector $\bm{\varkappa}$ need not be in the domain of the inverse operator~$(A^2e^{2xK} + \Gamma)^{-1}$ and thus we do not have representation~\eqref{eq:form.Q-positive-A}. Therefore, we shall make use of a different approximating sequence~$Q_n$. 

As we showed above, the operator $\Gamma$ is positive and compact and thus its finite-rank approximations $P_n\Gamma P_n$ converge to $\Gamma$ in norm. We set 
\[
	\gamma_n := \|\Gamma - P_n \Gamma P_n\|^{1/2};
\]
then $\gamma_n$ are positive for all $n\in\bN$, converge to $0$ as $n\to\infty$, and 
\begin{equation}\label{eq:form.Gamma}
	\Gamma  \le P_n \Gamma P_n + \gamma_n^2 I. 
\end{equation}
Set now 
\[
	Q_n(x) = \bigl\langle \bigl( (\gamma_n I + A^2)e^{2xK} +  P_n\Gamma P_n\bigr)^{-1} P_n \bm{\varkappa} , P_n \bm{\varkappa} \bigr\rangle
		=  \langle (\wt A_n^2 e^{2xK_n} + \Gamma_n)^{-1} \bm{\varkappa}_n , \bm{\varkappa}_n \rangle_{\bC^n},
\]
with $\wt A_n^2$ denoting the restriction of the operator $\gamma_n I + A^2$ onto $P_n H$. 
Then $q_n:=2Q_n'$ is a classical reflectionless potential corresponding to the negative eigenvalues $-\varkappa_1^2 < -\varkappa_2^2 < \dots < - \varkappa_n^2$ and norming constants $m_{j,n}:= \varkappa_j/(\gamma_n + \alpha^2_j)^{1/2}$, $j=1,2,\dots,n$. Therefore, the sequence $(q_n)_{n\in\bN}$ satisfies the assumption $A(\bm{\varkappa})$ of Section~\ref{sec:prelim} and by Lemma~\ref{lem:pre.condA} there exists a subsequence $(q_{n_k})_{k\in\bN}$ of  $(q_n)_{n\in\bN}$ that converges in the topology of the space $L_1(\mathbb{R})$ to some reflectionless potential $q_0\in\cQ_1$ of norm $\|q\|_1=4\sum_{j=1}^\infty\varkappa_j$.
We next observe that
$$  
	\varkappa_j(q)=\lim\limits_{k\to\infty}\varkappa_j(q_{n_k})=\varkappa_j, \qquad
	m_j(q)=\lim\limits_{k\to\infty} m_j(q_{n_k})=\varkappa_j/\alpha_j = m_j,
$$
so that the operator $T_{q_0}$ possesses the required negative eigenvalues and norming constants.
This also implies that for all  $x\in\mathbb{R}$,
$$   
\lim_{k\to \infty} Q_{n_k}(x)=-\frac12\lim_{k\to \infty} \int_x^\infty q_{n_k}(t)
= 	-\frac12 \int_x^\infty q_0(t)=: Q_0(x),
$$
and it remains to prove that $Q_0$ coincides with $Q$. 

We fix $x\in\bR$ and, for every $\eps>0$ and $n\in\bN$, set 
\begin{align*}
	Q_n(x;\eps) &:= \langle (\eps P_n + \wt A_n^2e^{2xK} + \Gamma_n)^{-1} \bm{\varkappa}_n ,\bm{\varkappa}_n \rangle_{\bC^n} \\
		&\,= \bigl\langle \bigl(\eps I + (\gamma_n I + A^2)e^{2xK} +  P_n\Gamma P_n\bigr)^{-1} P_n \bm{\varkappa} , P_n \bm{\varkappa} \bigr\rangle
\end{align*}
and observe that $Q_n(x;\eps) \le Q_n(x)$. Since the operators $\eps I + (\gamma_nI + A^2)e^{2xK} +  P_n\Gamma P_n$ are uniformly positive and converge, as $n\to\infty$, in the operator norm to the uniformly positive operator $\eps I + A^2e^{2xK} + \Gamma$, arguments similar to those of Case~1 give
\[
	\lim_{n\to\infty} Q_{n}(x;\eps) = Q_0(x;\eps) := \langle (\eps I + A^2e^{2xK} +  \Gamma)^{-1} \bm{\varkappa} , \bm{\varkappa} \rangle. 
\]
Passing to the limit over the subsequence $n_k$ in the inequality $Q_n(x;\eps) \le Q_n(x)$, we get that $Q_0(x;\eps) \le Q_0(x)$. By Proposition~\ref{pro:app.domain12}, 
\[
	\lim_{\eps \to 0+} Q_0(x;\eps) = \lim_{\eps \to 0+} \langle (\eps I + A^2e^{2xK} +  \Gamma)^{-1} \bm{\varkappa}, \bm{\varkappa} \rangle 
		= Q(x) 
\]
yielding the inequality $Q(x) \le Q_0(x)$. 

To prove the reverse inequality, we again fix $x\in\bR$ and observe that the operator~$e^{2xK}$ is uniformly positive. We denote by $\delta>0$ its lower bound; then by~\eqref{eq:form.Gamma}
\begin{align*}
	A^2e^{2xK} + \Gamma &\le  \gamma_n^2I + A^2e^{2xK} + P_n\Gamma P_n 
		\le  {\gamma_n^2}e^{2xK}/{\delta} + A^2e^{2xK} + P_n\Gamma P_n \\
		&\le (1 + \gamma_n/\delta) \bigl((\gamma_n I +A^2)e^{2xK} + P_n\Gamma P_n\bigr).
\end{align*}
Taking the inverses and recalling that 
\[
	Q_n(x) = \bigl\langle \bigl((\gamma_n I + A^2)e^{2xK} +  P_n\Gamma P_n\bigr)^{-1} P_n \bm{\varkappa} , P_n \bm{\varkappa}  \bigr\rangle
		\le \bigl\langle \bigl((\gamma_n I + A^2)e^{2xK} +  P_n\Gamma P_n\bigr)^{-1} \bm{\varkappa} , \bm{\varkappa} \bigr\rangle,
\]
we conclude by Proposition~\ref{pro:app.inverse} that 
\[
	Q_n(x) \le (1 + \gamma_n/\delta) Q(x).
\]
Passing to the limit over the subsequence $n_k$ results in the required inequality $Q_0(x) \le Q(x)$. The proof is complete.

\begin{remark}\label{rem:Marchenko}
	In the proof of Theorem~\ref{thm:form.refl}, we constructed a sequence of classical reflectionless potentials in $B(-\varkappa_1^2)$ converging to the potential~$q$ of~\eqref{eq:form.Q}--\eqref{eq:form.q} in the sense of uniform convergence on compact subsets of~$\bR$. Therefore, the potential~$q$ belongs to the set $\overline{B(-\varkappa_1^2)}$ of generalised reflectionless potentials in the sense of Marchenko~\cite{Mar91}.
\end{remark}

%

\section{Reflectionless potentials: uniqueness from three spectra}\label{sec:uniqueness}


In this section, we establish an analogue of the so-called inverse problem of reconstructing the potential from three spectra. Namely, along with the whole line Schr\"odinger operator $T_q$, we consider two half-line operators $T_q^+$ and $T_q^-$ on $\bR_+$ and $\bR_-$, respectively, generated by the differential expression~$-d^2/dx^2 + q$ and the Dirichlet boundary condition at $x=0$. It turns out that, similarly to a finite-interval case~\cite{GesSim,Piv,HM3sp}, the discrete spectra of these operators uniquely determine a generic (in the sense of Subsection~\ref{ssec:q-generic} below) reflectionless potential~$q$. In addition of being of independent interest, this result is essentially used to justify formula~\eqref{eq:3-sp} for norming constants. This formula, in turn, reduces the question on uniqueness of reflectionless Schr\"odinger operators with given spectral data to that with the three spectra. 

We start by explaining what a generic reflectionless potential is, then recall the Marchenko characterization~\cite{Mar91,Remling} of the reflectionless potentials in terms of the corresponding Borel measures and, finally, show that the three negative spectra recover uniquely the Weyl--Titchmarsh functions $m_\pm$ of $T_q$, and thus the potential~$q$ by the classical Borg--Marchenko theorem~\cite{Borg, Mar50, Mar52}. 

\subsection{Special vs generic potentials}\label{ssec:q-generic}

We call a potential $q \in \cQ_1$ \emph{special} if the negative spectra of $T_q$ and $T_q':=T_q^-\oplus T_q^+$ have nonvoid intersection. In other words, there is $-\varkappa_n^2$ that is also an eigenvalue of either $T_q^+$ or $T_q^-$. Observe that $e_\pm(\cdot, \pm i\varkappa_n; q)$ are the only (up to a multiplicative constant) solutions of the equation 
\[
- y'' + qy = -\varkappa_n^2 y
\]
that are integrable on $\bR_\pm$ and that $e_+(\cdot, i\varkappa_n; q)$ and $e_-(\cdot, -i\varkappa_n; q)$ are in fact proportional. Therefore, we see that then $e_\pm(0,\pm i\varkappa_n;q) = 0$ and thus the number $-\varkappa_n^2$ is an eigenvalue of both $T_q^+$ and $T_q^-$. 

Potentials in $\cQ_1$ that are not special are called \emph{generic}. It follows from the above considerations that $q$ is generic if and only if the equality $e_+(0,\lambda;q)=0$ holds for no $\lambda = i\varkappa_n$. 

For $\tau\in\bR$, set $q_\tau(x):= q(x+\tau)$ to be the left shift by $\tau$ of the potential $q$. The left shift does not change the spectrum of the operator $T_q$ but changes the spectrum of $T_q'$. 

\begin{lemma}
	There is at most countable set of $\tau\in\bR$ for which $q_\tau$ is special. 
\end{lemma}

\begin{proof}
	As was explained above, $q_\tau$ is special if and only if $e_+(0, i\varkappa_n;q_\tau)=0$ for some $n\in\bN$. It follows from uniqueness of the Jost solution that 
	\[
	e_+(x,i\varkappa_n;q_\tau) = e^{\varkappa_n\tau} e_+(x+\tau, i\varkappa_n;q),
	\]
	so that $e_+(0, i\varkappa_n;q_\tau)=0$ is equivalent to $e_+(\tau, i\varkappa_n;q)=0$. As a result, $q_\tau$ is special if and only if $\tau$ is a real zero of at least one Jost solution $e_+(\cdot, i\varkappa_n;q)$. It follows from the Sturm oscillation theorem that, for each~$n\in\bN$, the Jost solution $e_+(\cdot, i\varkappa_n;q)$ has at most $n$ real zeros. Being the union of such zeros over all $n\in\bN$, the set of those $\tau\in\bR$ for which $q_\tau$ is special is at most countable.
\end{proof}

As a corollary, we see that for any countable set~$S$ of potentials there are infinitely many $\tau\in\bR$ such that the left shift by $\tau$ makes all potentials in $S$ generic. In particular, there is no loss of generality to assume that the potential $q$ is generic; indeed, otherwise we just replace $q$ by appropriate $q_\tau$. This will not change the discrete spectrum $-\varkappa_n^2$ of $T_q$, while the norming constant $m_n$ will get multiplied by $e^{-\varkappa_n \tau}$ (see the proof of the above lemma).

\subsection{$m$-functions and representation of reflectionless potentials}\label{ssec:3sp-reflectionless}

In this subsection, we discuss some particulars of the approach to reflectionless Schr\"odinger operators originally due to Marchenko~\cite{Mar} and then elaborated by Hur, McBride, and Remling~\cite{Remling}. 

According to the classical Weyl theory, if a real-valued potential $q$ is locally integrable and is in the limit point case at $\pm\infty$, then for every non-real $z$ the equation
\[
	-y'' + qy = zy
\]
has unique (up to scalar factors) Weyl solutions $\psi_\pm(\cdot;z)$ that are square integrable at $\pm\infty$. The functions
\[
	m_\pm(z):= \pm\frac{\psi'_\pm (0,z)}{\psi_\pm (0, z)}
\]
are then called the \emph{Weyl--Titchmarsh} $m$-functions of the operators~$T^\pm_q$. Since for integrable $q$ the Jost solutions have the required integrability properties, the Weyl--Titchmarsh $m$-functions are then equal to 
\begin{equation}\label{eq:3sp-m}
	m_\pm(z) = \pm\frac{e'_\pm(0,\pm\sqrt{z};q)}{e_\pm(0,\pm\sqrt{z};q)}.
\end{equation}
The $m$-functions $m_\pm$ are known to be analytic in $\bC\setminus\bR$; moreover, they are Herglotz functions, i.e., map $\bC_\pm$ into $\bC_\pm$. 

Further, properties of the Jost solutions $e_+(\cdot,z;q)$ established in Section~\ref{sec:prelim} guarantee that $m_+$ has a meromorphic extension into the domain $\bC \setminus \bR_+$. The points $-\varkappa^2<0$ for which $e_+(0,i\varkappa;q) = 0$ are the poles of this extension; clearly, such points are the eigenvalues of the operator~$T_q^+$. Also, $m_+$ possesses finite limit values on $\bR_+$ from above and from below that are equal to 
\[
	m_+(k^2\pm i0) = \lim_{\eps\to0+} m_+(k^2\pm i\eps) = \frac{e'_+(0,\pm k;q)}{e_+(0,\pm k;q)}. 
\]
Likewise, $m_-$ can be extended meromorphically into the domain $\bC\setminus\bR_+$, with poles at the points~$-\varkappa^2<0$ that are eigenvalues of the operator~$T_q^-$; $m_-$ also possesses limit values on $\bR_+$ from above and from below equal to 
\[
	m_-(k^2\pm i0) = \lim_{\eps\to0+} m_-(k^2 \pm i\eps) = - \frac{e'_-(0,\mp k;q)}{e_-(0,\mp k;q)}. 
\] 
As $T_q$ is reflectionless, the scattering coefficient $b$ vanishes for all real non-zero $k$, whence 
\[
	e_+(x,k;q) = a(k) e_-(x,k;q)
\]
for such $k$ in view of~\eqref{eq:pre.ab}. 
Taking logarithmic derivatives of the above functions at $x=0$, we conclude that the $m$-functions $m_\pm$ of reflectionless Schr\"odinger operator~$T_q$ satisfy 
the relation
\begin{equation}\label{eq.2A}
	m_+(k\pm i0) = - \overline{m_-(k\pm i0)},  \qquad k\in\bR_+. 
\end{equation}
We note that the authors of~\cite{Remling} take relation~\eqref{eq.2A} as their starting point: they say that a real-valued function $q$ is a (generalized) reflectionless potential on a set $S\subset \bR_+$ if $q$ is locally integrable, the corresponding Schr\"odinger operator $T_q$ is in the limit point case at $\pm\infty$ and the $m$-functions $m_\pm$ satisfy~\eqref{eq.2A} almost everywhere on~$S$. 

Using the Schwarz reflection principle, one can combine $m_+$ and $m_-$ into single-valued functions $M_\pm$ that are defined on $\bC_+ \cup \bR_+ \cup \bC_-$, are analytic there, and take values in $\bC_\pm$; for non-real $z$, $M_\pm$ are defined via
\[
	M_+(z) = \begin{cases}
		m_+(z), &\qquad z \in \bC_+;\\
		-\overline{m_-(\overline{z})}, &\qquad z \in \bC_-;
	\end{cases} 
	\qquad 
	M_-(z) = \begin{cases}
		-\overline{m_-(\overline{z})}, &\qquad z \in \bC_+;\\
		m_+(z), &\qquad z \in \bC_-.
	\end{cases} 
\] 
Since $m_\pm$ can be extended to meromorphic functions over $\bC\setminus\bR_+$, we see that in fact $M_\pm$ are just two univalent branches of a meromorphic function~$M$ defined on a Riemannian two-sheeted manifold of $\sqrt{z}$. The change of variables $z \mapsto -z^2$ defines now a univalent function
\[
	n_q(z) := -M(-z^2)
\]
fixed by the condition that $n_q(k) = -m_+(-k^2)$ for large positive $k$; 
moreover, $n_q$ is a Herglotz function that is meromorphic in the whole complex plane outside the origin. The poles of $n_q$ are all real; more precisely, its positive poles $\xi$ come from the poles $-\xi^2$ of $m_+$, while its negative poles $\xi$ correspond to the poles $-\xi^2$ of $m_-$. In particular, $n_q$ is analytic outside a circle of radius $\varkappa_1$. 

Being a Herglotz function, $n_q$ possesses a special integral representation; using the known asymptotics of $n_q$, this can be specified as
\begin{equation}\label{eq.3A}
	n_q(z)=z +\int\frac{d\nu(t)}{t-z},  \qquad 
				z\in\bC\setminus\opn{supp}(\nu),
\end{equation}
for some discrete Borel measure $\nu = \nu_q$ of compact support. 

It turns out~\cite{Mar,Remling} that~\eqref{eq.3A} can be used to characterize all generalized reflectionless potentials. Namely, denote by $\cM$ the set of all finite non-negative Borel measures on~$\bR$ of compact support. If a potential~$q$ is reflectionless in the sense of~\eqref{eq.2A}, then (cf.~\cite{Mar, Kotani, Remling}) there exists a unique measure $\nu\in\cM$ such that the induced Herglotz function $n_q$ of~\eqref{eq.3A} is related to the Weyl--Titchmarsh functions~$m_\pm$ of~$T_q$ via
\begin{equation}\label{eq.m_pm-via-n}
	-m_+(-z^2;q)=   n_q(z) ,  \quad
	m_-(-z^2;q)=   n_q(-z) ,  \qquad    0<\arg z <\pi/2.                                                            
\end{equation}
Vice versa, for any measure $\nu\in \cM$, one introduces a Herglotz function $n_q$ via~\eqref{eq.3A} and defines the functions $m_\pm$ via~\eqref{eq.m_pm-via-n}; then~\cite{Remling} these are the $m$-function for some reflectionless Schr\"odinger operator $T_q$. Since $T_q$ is uniquely reconstructed from its Weyl--Titchmarsh $m$-functions by the classical Borg--Marchenko uniqueness theorem~\cite{Borg, Mar50, Mar52}, it follows that the mapping 
\begin{equation}\label{eq.5A}
	\cQ_1\ni q \mapsto \nu\in \cM
\end{equation}
is bijective~\cite{Mar,Remling}. In the next subsection, we show that $\nu$ for reflectionless $q\in\cQ_1$ is, in turn, uniquely determined by the discrete spectra of $T_q$, $T_q^+$, and $T_q^-$.

\subsection{Reconstruction of $n_q$ from three spectra}

We now turn to the question how the Herglotz function $n_q$ is related to the discrete spectra of the operators $T_q$ and $T_q^\pm$. 

Assume therefore that $q\in\cQ_1$ is a reflectionless potential and that $-\varkappa_n^2$ and $-\mu_n^2$ are eigenvalues of the operators $T_q$ and $T_q'$ respectively. We assume that the potential $q$ is generic in the sense of Subsection~\ref{ssec:q-generic}; by the minmax principle~\cite{Simon}, the two sequences then strictly interlace, viz.
\begin{equation}\label{eq:3sp.interlace}
-\varkappa_1^2 < -\mu_1^2 < -\varkappa_2^2 < - \mu_2^2 < \dots
\end{equation}
We recall that $\pm\mu_j>0$ if $-\mu_j^2$ is an eigenvalue of the operator $T_q^\pm$. Next, the corresponding measure~$\nu$ is discrete and bears point masses at the points $\mu_n$ constructed from the eigenvalues of the operators~$T_q^+$ and $T_q^-$ and, possibly, at the point $k=0$. Denoting these masses by~$d_n$, we see that 
\[
		\nu = \sum_{j=0}^\infty d_j \delta_{\mu_j}, \qquad
		n_q(z) = z - \frac{d_0}z + \sum_{j=1}^{\infty}\frac{d_j}{\mu_j-z}
\]
with $d_0=0$ if $\nu(\{0\})=0$ and $d_n>0$ for $n>0$; here $\delta_\mu$ is the Dirac point measure at a point~$\mu$. 

We next consider the auxiliary Herglotz function	
\begin{equation}\label{eq:3sp.Rmu}
	R_{\nu}(z):=1 + \int \frac{d\nu(t)}{t^2 - z} = 1 - \frac{d_0}{z} + \sum_{j=1}^{\infty} \frac{d_j}{\mu_j^2-z}
\end{equation}
and observe that for every $z \in \bC_+$ the following relation holds:
\begin{equation}\label{eq:3sp-R}
	2zR_{\nu}(z^2) = n_q(z) - n_q(-z) = - m_+(-z^2;q) - m_-(-z^2;q).
\end{equation}

\begin{lemma}\label{lem:3sp.R-zero}
	Assume that $q\in \cQ_1$ corresponds to the measure $\nu\in\cM$ and that $\xi\in\bR_+$ is such that $\pm\xi\notin\supp\nu$. Then the number $\lambda=-\xi^2$ is an eigenvalue of the operator~$T_q$ if and only if $\xi^2$ is a zero of the function $R_{\nu}$.
\end{lemma}

\begin{proof}
	As $\pm\xi\notin\supp\nu$, we find that 
	\[
	e_+(0,i\xi;q)\ne 0, \qquad e_-(0,-i\xi;q)\ne 0
	\]
	and thus by~\eqref{eq:3sp-R}
	\[
	-\frac{e'_+(0,i\xi;q)}{e_+(0,i\xi;q)} +\frac{e'_-(0,-i\xi;q)}{e_-(0,-i\xi;q)}=2\xi R_{\nu}(\xi^2).
	\]
	We next note that a number $-\xi^2<0$ is an eigenvalue of the operator $T_q$ if and only if the Jost solutions $e_+(\cdot,i\xi;q)$ and $e_-(\cdot,-i\xi;q)$ are linearly dependent, i.e., when $R_{\nu}(\xi^2)=0$. The proof is complete. 
\end{proof}

Combining the above results, we can justify the uniqueness in the problem of reconstructing a reflectionless $q\in\cQ_1$ from three spectra. 

\begin{theorem}\label{thm:uniq-3sp}
	Assume that $q$ is a generic generalized reflectionless potential in~$\cQ_1$. Then the negative eigenvalues of the operators $T_q$, $T_q^+$, and $T_q^-$ uniquely determine~$q$. 
\end{theorem}

\begin{proof}
Given such a~$q\in\cQ_1$, we denote by $-\varkappa_1^2<-\varkappa_2^2< \dots$ the negative eigenvalues of the operator $T_q$ and by $-\mu_1^2 < -\mu_2^2 < \dots$ the negative eigenvalues of $T_q'= T_q^+ \oplus T_q^-$. 
	
Assume also that there is another generic reflectionless potential~$\tilde q \in \cQ_1$ such that the corresponding Schr\"odinger operators $T_{\tilde q}$, $T^+_{\tilde q}$, and $T^-_{\tilde q}$ have the same negative eigenvalues as $T_q$, $T^+_q$, and $T^-_q$, respectively. 
We denote by $\nu$ and $\tilde \nu$ the measures constructed as explained in Subsection~\ref{ssec:3sp-reflectionless} and by $R$ and $\tilde R$ the corresponding functions of~\eqref{eq:3sp-R}. By Lemma~\ref{lem:3sp.R-zero}, the numbers $\varkappa_j^2$ are all the non-zero real zeros of both $R$ and $\tilde R$, while $\mu^2_j$ are their poles. Applying Theorem~\ref{thm:app-Herglotz-int-prod}, we conclude that 
\[
	R(z) = \prod_{j=1}^\infty \frac{z^2 - \varkappa_j^2}{z^2 - \mu_j^2} = \tilde R(z);
\]
the same theorem now implies that the masses $d_j$ and $\tilde{d}_j$ of $\nu$ and $\tilde\nu$ at the points $\xi_j$, $j \ge 0$, coincide and thus the measures $\nu$ and $\tilde\nu$ coincide as well.

As a result, the Weyl--Titchmarsh $m$-functions for the operators~$T_q^\pm$ and $T_{\tilde q}^\pm$ coincide. Since by the Borg--Marchenko  
uniqueness theorem~\cite{Borg, Mar50, Mar52} the potential of the Schr\"odinger operator is determined uniquely by the corresponding Weyl--Titchmarsh $m$-functions, we conclude that $q=\tilde q$. The proof is complete.
\end{proof}

%
\section{Uniqueness theorem from spectral data}\label{sec:uniq-scattering}
%

In this section, we shall prove that every integrable reflectionless potentials is uniquely determined by its spectral data, i.e., the following uniqueness theorem:

\begin{theorem}\label{thm:uniq-sp}
	There is at most one reflectionless potential $q$ in $\cQ_1$ for which $T_q$ has prescribed spectral data $(\bm{\varkappa}, \mathbf{m})$.
\end{theorem}
Recalling the existence result of Theorem~\ref{thm:form.refl}, we arrive at the following corollary:
\begin{corollary}\label{cor:one-to-one}
	Denote by $\ell_{1,+}(\bN)$ the set of all positive and strictly decreasing sequences in~$\ell_1(\bN)$ and by $\mathcal{R}_1$ the set of all reflectionless potentials in~$\cQ_1$ that are not classical. Then the scattering mapping 
	\[
	\mathcal{R}_1 \ni q \mapsto (\bm{\varkappa}, \mathbf{m}) \in \ell_{1,+}(\bN) \times \bR_+^{\bN}
	\]
	is one-to-one and onto. 
\end{corollary}
In view of Remark~\ref{rem:Marchenko}, we also conclude that 
\begin{corollary}\label{cor:Marchenko}
	$\mathcal{R}_1 \subset \widetilde{B}:= \cup_{\mu>0}\overline{B(-\mu^2)}$, i.e., every integrable reflectionless potential is a generalized reflectionless potential in the sense of Marchenko~\cite{Mar91}. 
\end{corollary}

Certainly, only the case of infinitely many negative eigenvalues is of interest, as otherwise such a $q$ is a classical Bargmann potential.  
Our approach consists in justifying first the three spectra formula~\eqref{eq:3-sp} that relates the norming constants of $T_q$ and negative eigenvalues of $T_q$, $T_q^+$, and $T_q^-$. Using this formula, we next show that the spectral data uniquely determine these three negative spectra, and then apply Theorem~\ref{thm:uniq-3sp}.

\subsection{Three spectra formula}
As usual, for a potential $q\in\cQ_1$, we denote by 
\[
	-\varkappa_1^2 < -\varkappa_2^2 < \dots 
\]
the finite or infinite sequence of negative eigenvalues of the operator $T_q$ and by 
\[
	-\mu_1^2 < -\mu_2^2 < \dots 
\]
the corresponding eigenvalue sequence for the operator $T_q'$. Without loss of generality, we assume that $q$ is generic, so that the negative spectra of $T_q$ and $T_q'$ have void intersection. Then the two sequences strictly interlace, viz.
\[
	-\varkappa_1^2 < -\mu_1^2 < -\varkappa_2^2 < - \mu_2^2 < \dots
\]
and each $-\mu_n^2$ is an eigenvalue of either $T_q^+$ or $T_q^-$ but not both; in the former case we take $\mu_n>0$ and in the latter case $\mu_n<0$. 

Next, we have the following formula for the norming constants $m_n$ corresponding to the eigenvalues~$-\varkappa^2_n$ of $T_q$:
\begin{equation}\label{eq:uniq.norm}
	m_n^{-2} := \int_{\bR} |e_+(x,i\varkappa_n;q)|^2 dx = i \dot{a}(i\varkappa_n;q)\frac{e_+(0,i\varkappa_n;q)}{e_-(0,-i\varkappa_n;q)};
\end{equation}
here $a$ is the standard scattering coefficient. Formally speaking, \eqref{eq:uniq.norm} was derived in~\cite{Mar} for $q$ in the Marchenko class, but the arguments only used existence of Jost solutions and thus can be applied to $q\in\cQ_1$ as well.

If $q = q_N$ is a classical reflectionless potential with $N$ negative eigenvalues, then the corresponding half-line Jost functions are equal to~\cite{Mar91} 
\begin{equation}\label{eq:3sp-Jost}
	e_\pm(0,\lambda;q_N) = \prod_{n=1}^N \frac{\lambda - i\mu_n}{\lambda \pm i\varkappa_n},
\end{equation}
and~\eqref{eq:uniq.norm} can be recast as
\begin{equation}\label{eq:3sp-clas}
	m_n^{-2} = i \dot{a}(i\varkappa_n;q_N) \prod_{l=1}^N \frac{\varkappa_n - \mu_l}{\varkappa_n+ \mu_l}.
\end{equation}
Recall also that the scattering coefficient $a(\cdot;q_N)$ is then a Blaschke product
\begin{equation}\label{eq:Blaschke-aN}
	a(z;q_N) = \prod_{n=1}^N \frac{z - i\varkappa_n}{z + i\varkappa_n},
\end{equation}
and thus formula~\eqref{eq:3sp-clas} relates the spectra of $T_q$, $T_q^+$, $T_q^-$, and the norming constants for~$T_q$. 

Our aim in this subsection is to prove that formula~\eqref{eq:3sp-clas} is valid also for reflectionless potentials $q\in\cQ_1$ with infinitely many negative eigenvalues. To this end, we first observe that $a(\cdot;q)$ is then given by the infinite Blaschke product
\begin{equation}\label{eq:Blaschke-a}
	a(z;q) = \prod_{n=1}^\infty \frac{z - i\varkappa_n}{z + i\varkappa_n};
\end{equation}
this follows e.g.\ from the continuity result of~\cite{HM-traces}. Next, with the sequence $(\mu_n)_{n\in\bN}$ constructed as explained above, we introduce the function
\begin{equation}\label{eq:Blaschke-B}
	B(z) = \prod_{n=1}^\infty \frac{z - i\mu_n}{z+ i\mu_n}. 
\end{equation}
The above product converges uniformly on compact subsets of $\bC\setminus \{-\mu_n\}_{n\in\bN}$ and can be written as the ratio $B_+(z)/B_-(z)$ of two Blaschke products $B_+$ and $B_-$, with
\begin{equation}\label{eq:uniq-sp.Bpm}
	B_\pm (z) := \prod_{n : \pm\mu_n>0}  \frac{z - i|\mu_n|}{z + i|\mu_n|}.
\end{equation}

\begin{theorem}[Three spectra formula]\label{thm:three-spectra}
Assume that $q$ is a generic potential in $\mathcal{R}_1$ and that $T_q$ has spectral data~$(\bm{\varkappa}, \mathbf{m})\in \ell_{1,+}\times \bR_+^{\bN}$. Construct the sequence $(\mu_n)_{n\in\bN}$ related to the negative spectrum of the operator~$T_q'$ and form the product $B$ as in~\eqref{eq:Blaschke-B}. Then for each $n\in\bN$, the right norming constant $m_n$ for the eigenvalue $-\varkappa_n^2$ of~$T_q$ satisfies the relation
\begin{equation}\label{eq:3-sp}
	m_n^{-2} = i \dot{a}(i\varkappa_n)B(i\varkappa_n).
\end{equation}
\end{theorem}

As in Section~\ref{sec:existence}, we start by constructing a sequence of classical reflectionless potentials converging to~$q$. Namely, for each $N\in\bN$, we denote by~$q_N$ the classical reflectionless potential associated with the sequences $(\varkappa_n)_{n=1}^N$ and $(\mu_n)_{n=1}^N$. We write $a_N$ for $a(\cdot;q_N)$ for short and also introduce the finite Blaschke products $B_{N,+}$ and $B_{N,-}$ as in~\eqref{eq:uniq-sp.Bpm} but using the first $N$ values, i.e., 
\begin{equation}\label{eq:uniq-sp.BpmN}
	B_{N,\pm} (z) := \prod_{n: \pm\mu_n>0}^N  \frac{z - i|\mu_n|}{z + i|\mu_n|},
\end{equation}
and set $B_N:=B_{N,+}/B_{N,-}$. Then equation~\eqref{eq:3sp-clas} for the right norming constant~$m_{n,N,+}$ of the operator $T_{q_N}$ corresponding to the eigenvalue $-\varkappa^2_n$ takes the form
\begin{equation}\label{eq:uniq.mnM}
		m_{n,N,+}^{-2} = i \dot{a}_N(i\varkappa_n)\frac{B_{N,+}(i\varkappa_n)}{B_{N,-}(i\varkappa_n)}.
\end{equation}

Application of Lemma~\ref{lem:pre.Blaschke} immediately gives the following result. 

\begin{lemma}
	As $N\to\infty$ and $n$ is fixed, the sequences $\dot{a}_N(i\varkappa_n)$ and $B_{N,\pm}(i\varkappa_n)$ converge respectively to $\dot{a}(i\varkappa_n)$ and $B_\pm(i\varkappa_n)$.
\end{lemma}

Since $q$ is generic, the sets $\{\varkappa_n\}_{n\ge1}$ and $\{|\mu_n|\}_{n\ge1}$ do not intersect; therefore, $B_-(i\varkappa_n)\ne0$ and the right-hand side of~\eqref{eq:uniq.mnM} has a finite non-zero limit as $N\to\infty$, so that 
\begin{equation}\label{eq:uniq.mnN-limit}
	\lim_{N\to\infty} m_{n,N,+}^{-2} 
		= \lim_{N\to\infty}i \dot{a}_N(i\varkappa_n)\frac{B_{N,+}(i\varkappa_n)}{B_{N,-}(i\varkappa_n)} 
		= i \dot{a}(i\varkappa_n)\frac{B_{+}(i\varkappa_n)}{B_{-}(i\varkappa_n)} \ne0.
\end{equation}
Therefore, the sequence $q_N$ satisfies assumption~$A(\bm{\varkappa})$ of Definition~\ref{def:pre.condA} and thus by Lemma~\ref{lem:pre.condA} there is a subsequence $q_{N_k}$ converging in $L_1(\bR)$ to a reflectionless potential~$q_0$. 

\begin{lemma}
  The above limit $q_0$ coincides with $q$.
\end{lemma}

\begin{proof}
By eigenvalue continuity, the negative eigenvalues of the operator $T_{q_0}$ coincide with the set $\{-\varkappa_n^2\}_{n\in\bN}$. 

We next show that the sequence $(\mu_n(q_0))_{n\ge1}$ constructed for the operator $T_{q_0}$ coincides with the sequence $(\mu_n)_{n\ge1}$ for $T_q$.
Recall that those $\mu_n$ that are positive (resp., negative) correspond to the zeros $i\mu_n$ of the right Jost function $e_+(0,\lambda;q)$ (resp. of the left Jost function $e_-(0,\lambda;q)$). As we already know, $e_\pm(0,\lambda;q_N)$ converge locally uniformly in $\lambda\in\bC_\pm$ to $e_\pm(0,\lambda;q_0)$; therefore, by the Rouch\'e theorem the zeros of $e_\pm(0,\lambda;q_N)$ in $\bC_\pm$ converge to those of $e_\pm(0,\lambda;q_0)$. The construction of the potentials $q_N$ now guarantees that $\mu_n(q_0) = \mu_n$ for every $n\in\bN$.

To sum up, the operators $T_{q_0}$, $T_{q_0}^+$ and $T_{q_0}^-$ have the same negative eigenvalues as the operators $T_q$, $T_q^+$ and $T_q^-$ respectively. By Theorem~\ref{thm:uniq-3sp}, $q_0 = q$, and the proof is complete.
\end{proof}

\begin{proofof}{Proof of Theorem~\ref{thm:three-spectra}.}
As the subsequence~$q_{N_k}$ constructed above converges to $q$ in the topology of $L_1(\bR)$, by Lemma~\ref{lem:pre.cont-norm} the norming constants corresponsign to the eigenvalue~$-\varkappa_n^2$ of the Schr\"odinger operator with potential $q_{N_k}$ converge to those of the operator~$T_q$. Therefore, 
\(
	\lim_{k\to\infty} m_{n,N_k,+} = m_n
\) 
for every fixed $n\in\bN$, which in view of~\eqref{eq:uniq.mnN-limit} completes the proof.
\end{proofof}

\subsection{Interpolation of ratios of Blaschke products} 

Before continuing with the proof of Theorem~\ref{thm:uniq-sp}, we establish one auxiliary interpolation result needed also for the existence of  Theorem~\ref{thm:3sp-exist}.

Recall that $\ell_{1,+}(\bN)$ stands for the set of all real sequences $\bm{\varkappa} =(\varkappa_j)_{j\in\bN}$ in $\ell_1(\bN)$ such that $\varkappa_j > \varkappa_{j+1} >0$ for all $j\in\bN$. For a fixed $\bm{\varkappa}\in \ell_{1,+}(\bN)$, we denote by $\Lambda(\bm{\varkappa})$ the set of real sequences $\bla=(\lambda_j)$ such that 
\[
\varkappa_j > |\lambda_j| > \varkappa_{j+1}, \qquad j\in\bN.
\]
Next, given a sequence $\bla \in \Lambda(\bm{\varkappa})$, we set 
\[
g_{\bla}(z) := \prod_{j=1}^\infty \frac{z - \lambda_j}{z+\lambda_j}.
\]
The function~$g_{\bla}$ is holomorphic on the set $\bC \setminus \bigl( \{0\}\cup\{-\lambda_j\}_{j\in\bN} \bigr)$ and has simple poles at the points of~$-\bla$.

\begin{theorem}\label{thm:interpolation}
	Assume that $\bm{\varkappa} \in \ell_{1,+}(\bN)$ is fixed. Then $\bla \in \Lambda(\bm{\varkappa})$ is uniquely determined by the values of the function $g_{\bla}$ at the points $\varkappa_j$, $j\in\bN$.
\end{theorem}

We reformulate this theorem in a more convenient language of entire functions and make repeated use of the fact that a canonical product of genus zero is of exponential type zero; see Theorem~7 in Sec.~4, Part~1 of Ch.~2 in~\cite{Young}. We recall that $P$ being of exponential type zero means that $\log^+ |P(z)| = o(|z|)$ as $|z|\to\infty$; in particular, the sum and the product of two functions of exponential type zero are of exponential type zero as well. The same is true of the ratio of two such functions whenever this ratio is an entire function.

\begin{proof}[Proof of Theorem~\ref{thm:interpolation}] For a sequence $\bla\in\Lambda(\bm{\varkappa})$, we introduce the canonical product
	\[
	h_{\bla}(w):= \prod_{j=1}^\infty \bigl(1-w\lambda_j\bigr);
	\]
	then by \cite[Thm~7, Ch.~II, Pt.~I, Sec.4]{Young} $h_{\bla}$ is an entire function of exponential type $0$  and
	\[
	g_{\bla}(z) = h_{\bla}(1/z)/h_{\bla}(-1/z). 
	\]
	
	Next assume that sequences $\bla = (\lambda_j)_{j\in\bN}$ and $\widetilde\bla = (\tilde\lambda_j)_{j\in\bN}$ in $\Lambda(\bm{\varkappa})$ are such that $g_{\bla}(\varkappa_j) = g_{\widetilde\bla}(\varkappa_j)$ for all $j\in\bN$; then one gets the equality
	\[
	h_{\bla}(w)h_{\tilde\bla}(-w) = h_{\tilde\bla}(w)h_{\bla}(-w)
	\]
	for $w=1/\varkappa_j$, $j\in\bN$. Introduce the entire function
	\[
	f(w):= h_{\bla}(w)h_{\tilde\bla}(-w) - h_{\tilde\bla}(w)h_{\bla}(-w);
	\]
	then $f$ is an odd function that vanishes at the points $w = \pm 1/\varkappa_j$ and $w=0$. We shall prove that $f \equiv 0$; then $\{\lambda_j\}_{j\in\bN} \cup \{-\tilde\lambda_j\}_{j\in\bN} = \{-\lambda_j\}_{j\in\bN} \cup \{\tilde\lambda_j\}_{j\in\bN}$ and thus $\bla = \widetilde\bla$.
	
	The function
	\[
	G(w):=\frac{f(w)}{h_{\bm{\varkappa}}(w)h_{\bm{\varkappa}}(-w)}
	\]
	is entire, of exponential type zero, and vanishes at $w=0$. On the imaginary axis $w=iy$, one gets the bound
	\[
	\Bigl| \frac{(1\pm iy\lambda_j)(1 \mp iy\tilde\lambda_j)}
	{(1-iy\varkappa_j)(1+iy\varkappa_j)}\Bigr| \le 1,
	\]
	so that $|G(iy)|\le 2$. By corollary to Theorem~22 in~\cite{Levin}, the function~$G$ is constant, and as $G(0)=0$, this constant is zero. The proof is complete.
\end{proof}	

	
\subsection{Proof of Theorem~\ref{thm:uniq-sp}} 

Given any element $(\bm{\varkappa},\mathbf{m})\in \ell_{1,+} \times \bR_+^\infty$, we showed in Section~\ref{sec:existence} that there exists a reflectionless potential $q\in \cQ_1$ whose spectral data coincide with that element; this $q$ is actually given explicitly by formulae~\eqref{eq:form.Q} and \eqref{eq:form.q}.

We now prove that such a $q$ is unique. For any reflectionless~$q\in\cQ_1$ with the given spectral data $(\bm{\varkappa},\mathbf{m})$, we construct the sequence $\mu_n$ generated by the negative eigenvalues of the operator $T_q'$. As the scattering coefficient $a(\cdot;q)$ is given by the Blaschke product~\eqref{eq:Blaschke-a} with $i\varkappa_n$, the three spectra formula~\eqref{eq:3-sp} shows that the product $B$ of~\eqref{eq:Blaschke-B} assumes known values at the points $i\varkappa_n$:
\[
	B(i\varkappa_n) = -i/\dot{a}(i\varkappa_n) m_n^{-2}. 
\]
The interpolation theorem (Theorem~\ref{thm:interpolation}) then implies that these values uniquely determine the sequence $(\mu_n)_{n\in\bN}$. By the three spectra uniqueness theorem (Theorem~\ref{thm:uniq-3sp}), the potential~$q$ is uniquely determined by the sequences $(\varkappa_n)_{n\in\bN}$ and $(\mu_n)_{n\in\bN}$, and the proof is complete.


\subsection{Characterization of three spectra}

As another application of the interpolation Theorem~\ref{thm:interpolation}, we can augment three spectra uniqueness Theorem~\ref{thm:uniq-3sp} with the existence result, thus giving a complete characterization of possible three spectra of integrable reflectionless potentials. 

\begin{theorem}\label{thm:3sp-exist}
	Assume that two real sequences $(\varkappa_n)_{n\in\bN}$ and $(\mu_n)_{n\in\bN}$ belong to $\ell_1(\bN)$ and satisfy the relations
	\begin{equation}\label{eq:3sp-interlace}
		\varkappa_1 > |\mu_1| > \varkappa_2 > |\mu_2| > \dots.
	\end{equation}
	Then there is a unique reflectionless potential $q$ in $\cQ_1$ such that the sets $\{-\varkappa_n^2\}_{n\in\bN}$ and 
	$\{-\mu_n^2 \mid \pm\mu_n>0\}$ coincide with negative spectra of operators~$T_q$ and $T_q^\pm$ respectively.
\end{theorem}

\begin{proof} We first construct functions $a$ of~\eqref{eq:Blaschke-a} and $B$ of~\eqref{eq:Blaschke-B} for the given sequences $(\varkappa_n)_{n\in\bN}$ and $(\mu_n)_{n\in\bN}$ and then use the three spectra formula~\eqref{eq:3-sp} to determine a sequence $\mathbf{m}=(m_n)_{n\in\bN}$. The interlacing property~\eqref{eq:3sp-interlace} guarantees that all the numbers $m_n$ are positive. By Theorem~\ref{thm:uniq-sp}, there is a unique reflectionless potential~$q\in\cQ_1$ with the spectral data $(\bm{\varkappa},\mathbf{m})\in \ell_{1,+} \times \bR_+^\infty$. 
	
Now, we use the sequence $(\mu_n(q))$ generated by the operators~$T_q^+$ and $T_q^-$ to construct the Blaschke product~\eqref{eq:Blaschke-B}, 
\[
	B(z;q) = \prod_{n=1}^\infty \frac{z - i\mu_n(q)}{z+ i\mu_n(q)}. 
\]
The three spectra formula~\eqref{eq:3-sp} for the operator~$T_q$ shows that, for every $\varkappa_n$, we have $B(i\varkappa_n) = B(i\varkappa_n;q)$, and then the interpolation Theorem~\ref{thm:interpolation} yields the equalities $\mu_n(q) = \mu_n$ for every $n\in\bN$.  Therefore, the potential $q \in \cQ_1$ is such that the corresponding Schr\"odinger operators possess the required negative spectra.

Uniqueness of such a $q$ follows again from the interpolation Theorem~\ref{thm:interpolation} and the uniqueness Theorem~\ref{thm:uniq-sp}. The proof is complete. 
\end{proof}

\subsection{Some further remarks}

Here we give a few relations that are well known for the classical reflectionless potentials and justify their extension to the more general class of potentials in~$\mathcal{R}_1$. 

As a first example, the Jost functions $e_\pm(0,\lambda;q)$ for a classical reflectionless potential~$q$ is given by~\eqref{eq:3sp-Jost}. Assume now that $q \in \cQ_1$
is a (generic) generalized reflectionless potential with negative spectrum $-\varkappa_1^2 < -\varkappa_2^2 < \dots$, and let $-\mu_1^2 < - \mu_2^2 < \dots$ be negative eigenvalues of the operator~$T_q'$, with the standard convention on the signs of $\mu_k$. In the proof of Theorem~\ref{thm:uniq-sp}, we constructed a sequence $q_n$ of classical reflectionless potentials such that $T_{q_n}$ has negative eigenvalues $\{-\varkappa_1^2, - \varkappa_2^2, \dots ,- \varkappa_n^2\}$ and $T'_{q_n}$ has negative eigenvalues $\{-\mu_1^2, -\mu_2^2, \dots, -\mu_n^2\}$ and such that a subsequence $(q_{n_k})_{k\in\bN}$ of $(q_n)_{n\in\bN}$ converges to~$q$ in the topology of the space~$L_1(\bR)$. Passing to the limit in~\eqref{eq:3sp-Jost} over that subsequence and recalling continuity of the Jost solutions (Lemma~\ref{lem:pre.Jost-bounds}), we get
\[
	e_\pm(0,\lambda;q) = \prod_{n=1}^\infty \frac{\lambda - i\mu_n}{\lambda \pm i\varkappa_n}.
\]
Observe that if $q$ is not generic, then $\pm\mu_n = \varkappa_n$ for some $n\in\bN$, and some extra information is needed to compensate those missing terms in the above product, cf.~\cite{Mar91}.

In the same manner, for every~$\tau\in\bR$, for which the shifted potential~$q_\tau$ is generic, we get
\[
	e_\pm(\tau,\lambda;q) = \prod_{n=1}^\infty \frac{\lambda - i\mu_n(\tau)}{\lambda \pm i\varkappa_n},
\]
where $\mu_n(\tau)$ are constructed as before but for the splitting of the operator $T_q$ by the Dirichlet boundary condition at the point $x=\tau$. 

As a third example, we take the logarithmic derivative in $\tau$ at $\tau=0$ of the above expression for~$e_+(\tau, \lambda; q)$ to get the formula for the Weyl--Titchmarsh $m$-function in the upper half-plane (cf.~\eqref{eq:3sp-m}),
\[
	m_+(\lambda) = \frac{d}{d\tau}\log e_+(\tau,\lambda;q)\Bigl|_{\tau=0} = \sum_{n=1}^\infty \frac{-i\mu_n'(0)}{\lambda - i\mu_n}, 
\]
which is an analogue of the known formula in the classical reflectionless case~\cite{Mar91}. 

%

\section{Integrable soliton solutions of the KdV equation}\label{sec:KdV}

In this section, we justify the formula for the generalized soliton solution of the Korteweg--de Vries equation
\begin{equation}\label{eq:KdV.eq}
	u_t - 6u_xu + u_{xxx} =0,
\end{equation}
whose value at $t=0$ is the given integrable reflectionless potential~$q(x) = u(x,0)$, $x\in\bR$. That formula is suggested by the inverse scattering transform approach to the KdV equation~\cite{GGKM}; namely, if we regard $q_t:=u(\cdot,t)$ as the potential of the Schr\"odinger operator $T_{q_t}$, then the scattering data (the reflection coefficient $r(\cdot,t)$, the negative eigenvalues $-\varkappa_n^2(t)$, and the corresponding norming constants $m_n(t)$) satisfy the following relations:
\[
	r(k,t) = e^{8ik^3t}r(k,0), \qquad -\varkappa_n^2(t) = -\varkappa^2_n(0), \qquad m_n(t) = e^{8\varkappa_n t}m_n(0).
\]
Therefore, if we denote by $K(t)$ and $A(t)$ the diagonal operators in $\ell_2(\bN)$ constructed for the Schr\"odinger operator~$T_{q_t}$ as explained in Subsection~\ref{ssec:formula}, then $K(t)\equiv K(0)$ and $A(t) = A(0)\exp\{-8tK(0)^3\}$. This motivates the following statement.

\begin{theorem}\label{thm:KdV}
	Assume that $q\in\mathcal{R}_1$ is a generalized reflectionless potential and denote by $-\varkappa_1^2 < - \varkappa_2^2 < \dots$ and $m_1, m_2,\dots$ respectively the negative eigenvalues and the norming constants of the corresponding Schr\"odinger operator~$T_q$. Further, introduce in the Hilbert space~$H=\ell_2(\bN)$ the diagonal operators  $K=\diag\{\varkappa_1,\varkappa_2,\dots\}$ and $A = \diag\{\alpha_1,\alpha_2, \dots\}$, with $\alpha_n := \varkappa_n /m_n$, the trace class operator~$\Gamma$ via~\eqref{eq:form.Gamma} and a vector $\bm{\varkappa}:= (\varkappa_1,\varkappa_2,\dots)$. Then the function
	\begin{equation}\label{eq:KdV.sol}
		u(x,t) = 2 \frac{d}{dx} \|(A^2 e^{2xK-8tK^3} + \Gamma)^{-1/2}\bm{\varkappa}\|^2
	\end{equation}
	is well defined and gives a classical solution of the Korteweg--de Vries equation~\eqref{eq:KdV.eq} with initial data $u(x,0) = q(x)$. 
\end{theorem}
	
\def\wt{\widetilde}

To simplify the calculations, we introduce an auxiliary function
\begin{equation}\label{eq.2x}
	\varphi(x,t):= \|(A^2e^{xK-tK^3} +\Gamma)^{-1/2}\bm{\varkappa}\|^2, \qquad x,t\in\bR,
\end{equation}
which is related to $u(x,t)$ by the formula $u(x,t) = 2\frac{d}{dx} \varphi(2x,8t)$. Denote also by $\Omega_\nu$ the cylindrical domain of $\bC^2$ of the form
\[
	\Omega_\nu:= \{(z,\zeta)\in \bC^2\mid |\myIm  z| < \nu/(2\kappa_1), |\myIm \zeta| < \nu/(2\kappa_1^3)\}.
\]

\begin{lemma}\label{lem:KdV}
	The function $\varphi$ is a solution of the nonlinear equation
	\begin{equation}\label{eq.3x}
		v_t - 3(v_x)^2+v_{xxx}=0.
	\end{equation}
	Moreover, it admits a holomorphic continuation in the cylindrical domain $\Omega_{\pi/2}$ and satisfies there the bounds 
	\begin{equation}\label{eq.4xx}
		|\varphi(z,\zeta)|\le  \frac{|\varphi(\myRe z,\myRe \zeta)|}{\cos\nu} \le \frac{2}{\cos\nu}\sum_{j\ge1}\kappa_j
\end{equation}
whenever $(z,\zeta) \in \Omega_\nu$ for some $\nu \in (0,\pi/2)$.
\end{lemma}

We start with the following elementary observation that will allow us to consider only some special cases. Let $\mathscr{B}_+(H)$ and $\mathscr{S}_+(H)$ denote the sets of all bounded positive operators and self-adjoint positive operators in the Hilbert space~$H$, respectively. Assume that the operator $A$ can be written as $A=A_1A_2$, where the factors $A_1\in\mathscr{B}_+(H)$ and $A_2\in\mathscr{S}_+(H)$ commute with $K$ and with each other. If the operator $A_2$ is also uniformly positive, then
	\begin{equation*}
	\varphi(x,t) 
	  	= \|(A^2_1e^{xK-tK^3} +\wt\Gamma)^{-1/2}\wt\bkappa\|^2,
	\end{equation*}
	with $\wt\Gamma:=A_2^{-1}\Gamma A_2^{-1}$ and $\wt\bkappa:= A_2^{-1}\bkappa$. Since the operators $K$ and $A_2^{-1}$ commute, we get
	\begin{equation*}
	K\wt\Gamma+ \wt\Gamma K=\langle\,\cdot\,, \wt\bkappa\rangle\wt\bkappa.
	\end{equation*}	
Therefore, the case $A\in\mathscr{S}_+(H)$ can be reduced to that of $A\in\mathscr{B}_+(H)$, 
	and the case of a uniformly positive $A$ to that of $A=I$.
	
\begin{proofof}{Proof of Lemma~\ref{lem:KdV}: Step~1.}
Consider first the simpler case of a uniformly positive operator~$A$. As explained above, this can be reduced to $A=I$, which we assume for what follows. For convenience, we introduce auxiliary functions
\[ 
	E(x,t):=e^{xK-tK^3}, \quad    
	V(x,t):=(E(x,t) +\Gamma)^{-1}, \quad
	B(x,t):=-V(x,t)E(x,t)K.
\]
Observe that the operator $e^{xK-tK^3} + \Gamma$ is then uniformly positive, so that $\varphi$ can be written as 
\[
	\varphi(x,t):= \bigl\langle(e^{xK-tK^3} +\Gamma)^{-1}\bkappa, \bkappa\bigr\rangle = \langle V \bkappa , \bkappa \rangle.
\]
Direct differentiation produces
\begin{equation}\label{eq.5x}
	\begin{split}
		E_x=&EK, \qquad E_t=-EK^3,   \\
		V_x=&BV, \qquad   V_t=-BK^2 V, 
	\end{split}
\end{equation}
and 
\begin{equation}\label{eq.6x}
	B_x=-V_xEK - VE_xK=-BVEK-VEK^2= B^2+BK .
\end{equation}
Using~\eqref{eq.5x} and \eqref{eq.6x}, we next find that
$$
V_{xx}=(BV)_x=B_xV+BV_x= (B^2+BK)V +B^2V=B(2B+K)V
$$
and, differentiating once again, that
\begin{multline} \label{eq.7x}
		V_{xxx}=B_x(2B+K)V +2BB_xV+ B(2B+K)V_x=\\
		= (B^2+BK)(2B+K)V +2B(B^2+BK)V +B(2B+K)BV=\\
		= B[6B^2+3(KB+BK)+K^2)]V.
\end{multline}

Denote by $R$ the linear functional $R:\, H \to \bC$ defined via
\[
	R  :=\langle\,\cdot \,, \bkappa\rangle; 
\]    
then $K\Gamma+ \Gamma K=R^*R$ and $\varphi=RVR^*$. 
Using equalities \eqref{eq.7x} and \eqref{eq.5x}, we arrive at the relation
\begin{equation}\label{eq.8x}
	\varphi_t+\varphi_{xxx} = R(V_t+V_{xxx})R^* 
			= 3RB[2B^2 + KB + BK)]VR^*.
\end{equation}
Observe that 
\[
	KV + VK = V(E+\Gamma)KV + VK(E+\Gamma)V = -2BV + VR^*RV
\]
and, further, that 
\[
	KB + BK = -(KV+VK)EK = 2BVEK - VR^*RVEK = -2B^2 + VR^*RB.
\]
Combining this with \eqref{eq.8x}, we conclude that
$$
	\varphi_t+\varphi_{xxx}=R(V_t+V_{xxx})R^*= 3RBVR^*RBVR^*= 3(RV_x R^*)^2= 3(\varphi_x)^2,
$$
i.e., that
$$
	\varphi_t - 3(\varphi_x)^2 + \varphi_{xxx}=0.
$$

We next prove that the formula
\begin{equation*}
	\varphi(z,\zeta):= \bigl\langle(e^{zK-\zeta K^3} +\Gamma)^{-1}\bkappa,\bkappa\bigr\rangle
\end{equation*}
defines a function that is holomorphic in the domain~$\Omega_{\pi/2}$. To this end, it suffices to show that, for every $(z,\zeta)\in \Omega_{\pi/2}(K)$, the operator $ e^{zK-\zeta K^3} +\Gamma$ is boundedly invertible. We fix an arbitrary 
$(z,\zeta)\in \Omega_{\pi/2}(K)$ with $z=x+iy$, $\zeta= t+i s$, set
$$   L:=y K - s K^3,  \qquad M=\frac12(xK-t K^3),
$$
and note that the operators $L$ and $M$ are self-adjoint and commuting. Direct calculations show that 
\begin{equation*}
	e^{zK-\zeta K^3} +\Gamma = e^{M}e^{i L}e^{M} +\Gamma=e^{M}(\cos L)e^{M} +\Gamma
						+ie^{M}(\sin L)e^{M};
\end{equation*}
since $\|L\| < \pi/2$ by the definition of the set $\Omega_{\pi/2}$, we conclude that the operators $\cos L$ and $C:=e^{M}(\cos L)e^{M} +\Gamma $ are positive and invertible in the algebra $\mathscr{B}(H)$ of all bounded operators in $H$. 
Therefore, 
\[ 
	e^{zK-\zeta K^3} +\Gamma=C^{1/2}(I +iF)C^{1/2}
\]
with bounded self-adjoint $F:= C^{-1/2}e^{M}(\sin L)e^{M}C^{-1/2}$; as the operator $(I +iF)$ is invertible and $\|(I +iF)^{-1}\|\le 1$, we see that the operator $e^{zK-\zeta K^3} +\Gamma$ is also invertible and 
\begin{equation*}\label{eq.11x}
(e^{zK-\zeta K^3} +\Gamma)^{-1}=C^{-1/2} (I +iF)^{-1}  C^{-1/2}.
\end{equation*}
As a result, the function $\varphi$ is proved to be holomorphic in the domain $\Omega_{\pi/2}$ and satisfies there the bound 
$$
	|\varphi(z,\zeta)|\le \|RC^{-1/2}\| \|C^{-1/2} R^*\|=\|RC^{-1}R^*\|.
$$ 

To refine the above bound, we first observe that, by the spectral theorem, $(\cos L)\ge (\cos \|L\|)I$, so that
$$    
	C=e^{M}(\cos L)e^{M} +\Gamma \ge (\cos \|L\|)(e^{2M}+\Gamma)
$$
and, consequently,
$$  
	RC^{-1}R^*\le (\cos \|L\|)^{-1}R(e^{2M}+\Gamma)^{-1}R^* = (\cos \|L\|)^{-1} \varphi(x,t).
$$
Recalling that
$$   
	\varphi(x,t):= \|(e^{xK-tK^3} +\Gamma)^{-1/2}\bkappa\|^2\le \|\Gamma^{-1/2}\bkappa\|^2,
$$ 
we derive the required bound
\begin{equation*}
	|\varphi(z,\zeta)| \le  \frac{|\varphi(x,t)|}{\cos \|L\|} \le  \frac{\|\Gamma^{-1/2}\bkappa\|^2}{\cos \|L\|}
		\le \frac{\|\Gamma^{-1/2}\bkappa\|^2}{\cos \nu}
\end{equation*}
whenever $(z,\zeta) \in \Omega_\nu(K)$. 
\end{proofof}

\begin{proofof}{Proof of Lemma~\ref{lem:KdV}: Step~2.}
We discuss now the case of a general $A\in \mathscr{S}_+(H)$. As explained above, without loss of generality the operator~$A$ can be considered bounded. For an arbitrary $\varepsilon>0$, we set 
	$A_\varepsilon := (A^2+ \varepsilon I)^{1/2}$, 
	$B_\varepsilon:=(A^2+ \Gamma +\varepsilon I)^{1/2}$
and 
$$  
	\varphi_\varepsilon(z,\zeta)	
		:= \bigl\langle(A_\varepsilon^2e^{zK-\zeta K^3} +\Gamma)^{-1}\bkappa, \bkappa\bigr\rangle, 
			\qquad  (z,\zeta)\in\Omega_{\pi/2}.
$$
By Step 1 of the proof, for every $\varepsilon>0$  the function $\varphi_\varepsilon$ is holomorphic in $\Omega_{\pi/2}$ and satisfies there \eqref{eq.4xx} for all $(z,\zeta)\in \Omega_\nu$ with $\nu\in(0,\pi/2)$. Moreover, its restriction onto~$\mathbb{R}^2$ is a solution of equation~\eqref{eq.3x}. To complete the proof, it suffices to show that, as $\varepsilon\to +0$, the functions $\varphi_\varepsilon$ converge to $\varphi$ uniformly on compact sets of $\Omega_{\pi/2}$. 

By virtue of \eqref{eq.4xx}, the partial derivatives $(\varphi_\varepsilon)_z $ and $(\varphi_\varepsilon)_{\zeta} $ are uniformly bounded on every compact set of $\Omega_{\pi/2}$, so that it remains to prove that, as $\varepsilon\to +0$, the functions $\varphi_\varepsilon$ converge pointwise in $\Omega_{\pi/2}$, i.e., that for every $(z,\zeta)\in \Omega_{\pi/2}$, there exists a finite limit  
\begin{equation*}\label{eq.12x}
	\lim\limits_{\varepsilon\to +0} \bigl\langle(A_\varepsilon^2e^{zK-\zeta K^3} +\Gamma)^{-1}\bkappa, \bkappa\bigr\rangle.
\end{equation*}

Fix an arbitrary $(z,\zeta)\in \Omega_{\pi/2}$. Using the notations of Step~1 of the proof, we see that
\begin{multline*} 
	A^2_\varepsilon e^{zK-\zeta K^3} +\Gamma
		= A_\varepsilon e^{M}e^{i L}e^{M}A_\varepsilon +\Gamma
		= A_\varepsilon  e^{M}(\cos L)e^{M}A_\varepsilon  +\Gamma +i A_\varepsilon e^{M}(\sin L)e^{M}A_\varepsilon
\end{multline*}
and, therefore, that
\begin{equation*}\label{eq.13x}
	\bigl\langle(A_\varepsilon^2e^{zK-\zeta K^3} +\Gamma)^{-1}\bkappa, \bkappa\bigr\rangle=
			\bigl\langle (W_\varepsilon +iV_\varepsilon)^{-1} B^{-1}_\varepsilon \bkappa, B^{-1}_\varepsilon\bkappa\bigr\rangle,
\end{equation*}
where
\begin{equation*}\label{eq.14x}
\begin{split}
	W_\varepsilon &= B^{-1}_\varepsilon A_\varepsilon  e^{M}(\cos L)e^{M}A_\varepsilon
			B^{-1}_\varepsilon + B^{-1}_\varepsilon \Gamma B^{-1}_\varepsilon,   \\
	V_\varepsilon &= B^{-1}_\varepsilon A_\varepsilon e^{M}(\sin L)e^{M}A_\varepsilon B^{-1} _\varepsilon.
\end{split}
\end{equation*}

Since $\bkappa\in\dom \Gamma^{-1/2}=\operatorname{ran}(\Gamma^{1/2})$ and $A^2+\Gamma\ge\Gamma$, in view of Proposition~\ref{pro:app.inverse} we get $\bkappa\in \dom (A^2+\Gamma)^{-1/2}$; thus, according to Proposition~\ref{pro:app.domain12},
$$   
	\lim\limits_{\varepsilon\to +0} B^{-1}_\varepsilon \bkappa
		=(A^2+\Gamma)^{-1/2}\bkappa.    
$$
Therefore, it suffices to show that there exists the strong limit
	$\slim\limits_{\varepsilon\to+0} (W_\varepsilon +iV_\varepsilon) ^{-1}$.
Taking into account Lemma~\ref{le.42y}, we derive existence of the limits
$$  
	\slim\limits_{\varepsilon\to+0}B^{-1}_\varepsilon A_\varepsilon, \qquad
	\slim\limits_{\varepsilon\to+0} A_\varepsilon B^{-1}_\varepsilon,
$$
$$  
	\slim\limits_{\varepsilon\to+0}B^{-1}_\varepsilon \Gamma^{1/2}, \qquad
	\slim\limits_{\varepsilon\to+0} \Gamma^{1/2}B^{-1}_\varepsilon.
$$
so that the limit $\slim\limits_{\varepsilon\to+0}(W_\varepsilon +iV_\varepsilon)$ exists as well.

Since the operators $e^{M}$ and $\cos L$ are uniformly positive, for some $c\in (0,1)$ we have
$$  
	e^{M}(\cos L)e^{M} \ge cI, 
$$
so that
$$  
	W_\varepsilon = B^{-1}_\varepsilon (A_\varepsilon  e^{M}(\cos L)e^{M}A_\varepsilon 
			+ \Gamma) B^{-1}_\varepsilon \ge cB^{-1}_\varepsilon (A^2_\varepsilon 
			+ \Gamma) B^{-1}_\varepsilon
				=cI. 
$$

Since the operator $V_\varepsilon$ is self-adjoint, we see that 
$$   
	\myRe(W_\varepsilon+ iV_\varepsilon)\ge cI
$$
so that
$$      
	\|(W_\varepsilon+ iV_\varepsilon)^{-1}\| \le c^{-1}, \qquad \varepsilon>0.       
$$
Applying now Lemma~\ref{le.43y}, we conclude that the limit
$$	
	\slim\limits_{\varepsilon\to+0} (W_\varepsilon +iV_\varepsilon) ^{-1}
$$
exists, and this finishes the proof. 
\end{proofof}

\begin{proofof}{Proof of Theorem~\ref{thm:KdV}}
Setting $u(x,t) = 4 \varphi'_x(2x,8t)$ and differentiating, we find that, indeed,
\begin{multline*}
	u_t - 6 u_tu + u _{xxx}=
	[32\varphi_{xt} - 96 \varphi_{xx}\varphi_{x}+ 32\varphi_{xxx}](2x,8t)\\
	= 32 \frac{\partial} {\partial x}[\varphi_{t}  - 3(\varphi_x)^2 +\varphi_{xxx}] (2x,8t)=0
\end{multline*}
as required. The fact that $u$ is a classical solution follows from the properties of the function~$\varphi$.
\end{proofof}

\begin{remark}\label{re.} Along with $\varphi$, the solution $u$ of the KdV equation allows ananlytic continuation in the cylindrical domain $\Omega_{\pi/2}$; moreover, it satisfies the bound
	\[
			|u(z,\zeta)| \le  \frac{2}{\cos\nu}\sum_{j\ge1}\kappa_j
	\]
whenever $(z,\zeta)\in \Omega_{\nu}$ with $\nu\in(0,\pi/2)$. 	
\end{remark}

\medskip

\noindent\emph{Acknowledgement}
	The authors thank Prof.~Gesztesy for stimulating discussions and literature comments. The research was partially supported by Ministry of Education and Science of Ukraine, grant no.~0118U002060. The first author acknowledges support of the Centre for Innovation and Transfer of Natural Sciences and Engineering Knowledge at the University of Rzesz\'ow.

\appendix
\section{Some auxiliary results}\label{sec:app}

\subsection{The auxiliary result on one class of meromorphic functions}\label{ssec:app.Herglotz}

\begin{definition}\label{def:Herglotz} 
	Assume that a function $\varphi:\Omega\to\bC$ is analytic in a domain $\Omega\subset\bC$ containing the upper half-plane $\bC_+:=\{z\in\bC \mid |\myIm z>0\}$.
	We say $\varphi$ is a Herglotz function in $\bC_+$ if $\myIm\varphi(z)\ge 0$ for $z\in \bC_+$.
\end{definition}
Herglotz functions are also called Nevanlinna functions. Every Herglotz function can be represented in $\bC_+$ as (cf.~\cite{Nev})
\begin{equation}\label{eq.1E}
	\varphi(z) = az+b +\int_{\bR} \left(\frac{1}{t-z}-\frac{t}{1+t^2}\right) d\nu(t), \qquad z\in\bC_+,
\end{equation}
where $a\ge 0$, $b\in \bR$, and $\nu$ is a non-negative Borel measure on $\bR$ that is bounded on compact subsets of~$\bR$ and satisfies the condition
$$    
\int_{\bR} \frac{d\nu(t)}{1+t^2} <\infty.
$$

We denote by $\cM$ the set of all Borel measures $\nu$ of the form
\[
\nu  = \sum_{j=0}^{\infty} d_j \delta_{\xi_j},
\]
where $\delta_{\xi}$ denotes the Dirac delta-function centred at $\xi\in\bR$, $(\xi_j)_{j=1}^{\infty}$ is a strictly decreasing sequence of positive numbers converging to~$0$, $(d_j)_{j=1}^{\infty}$ is a summable sequence of positive numbers, $\xi_0=0$, and $d_0\ge 0$.

With every $\nu\in\cM$, we associate the Herglotz function
\[
	\phi_\nu(z) = 1 + \int_{\bR}\frac{d\nu}{t-z} = 1 + \sum_{j=0}^{\infty} \frac{d_j}{\xi_j - z}.
\]
The sum on the right-hand side converges uniformly on compact subsets of $\bC\setminus \{\supp \nu\}$ and defines a meromorphic function on the set $\bC\setminus\{0\}$ with poles at the points~$\xi_k$, $k\in\bN$. Being a Herglotz function, $\phi_\nu$ does not vanish on $\bC_\pm$ and strictly increases on each real interval not containing any pole. As a result, $\phi_\nu$ possesses a single zero $\eta_{k}$ in every interval $(\xi_{k},\xi_{k-1})$, $k=2,3, \dots$. Also, there is a single zero $\eta_1$ in the interval $(\xi_1,\infty)$. Observe also that $d_j$ with $j>0$ is the residue of the function $\phi_\nu$ at the pole $z=\xi_j$ and thus is uniquely determined by~$\phi_\nu$.

Next, we denote by~$\Lambda$ the set of all strictly decreasing sequences $\bla=(\lambda_n)_{n\in\bN}$ of positive numbers converging to zero.
With every $\bla\in\Lambda$, we associate the product 
\begin{equation}\label{eq.5E}
\psi_{\bla}(z):=\prod\limits_{n=1}^\infty\frac{z-\lambda_{2n-1}}{z-\lambda_{2n}}.
\end{equation}
In view of the relation 
\[ 
\frac{z-\lambda_{2n-1}}{z-\lambda_{2n}}
= 1+ \frac{\lambda_{2n}-\lambda_{2n-1}}{z-\lambda_{2n}}
\]
and convergence of the series 
\[
\sum_{n\ge1}|\lambda_{2n}-\lambda_{2n-1}|,
\]
the above product converges uniformly on compact subsets of $\bC$ not intersecting with the set
\(     \{\lambda_{2n}\mid n\in\bN\}\cup\{0\} 
\)
and thus defines a function that is meromorphic in~$\bC\setminus\{0\}$. The points $\lambda_{2n}$ are the poles of $\psi_{\bla}$, while $\lambda_{2n-1}$ are its zeros.

The main result of this subsection is that the set of functions $\phi_\nu$ and $\psi_{\bla}$ of the above form coincide. Namely, the following statements hold true.

\begin{theorem}\label{thm:app-Herglotz-int-prod}
	For every $\bla\in\Lambda$, there exists a unique measure $\nu\in\cM$ such that $\psi_{\bla}=\phi_\nu$. Conversely, 
	for every $\nu \in \cM$, there exists a unique sequence $\bla\in\Lambda$ such that $\phi_\nu = \psi_{\bla}$.
\end{theorem}

\begin{proof} 
	Uniqueness in both parts is straightforward, so that only existence need to be justified. We observe that  	
	a similar statement for meromorphic functions is proved in \cite[p.~399]{Levin} and is based on Krein's idea; see also~\cite{GesSim}. We slightly modify the arguments for the current setting.

	Take an arbitrary sequence $\bla \in \Lambda$. Observe that the argument
	$$
	\alpha_n:=\arg\left(\frac{z-\lambda_{2n-1}}{z-\lambda_{2n}}\right)= \arg{(z-\lambda_{2n-1})} -\arg{(z-\lambda_{2n})}
	$$
	is equal to the angle at which the interval $[\lambda_{2n-1},\lambda_{2n}]$ is seen from the point $z\in\bC_+$, so that 
	$$   
	\arg \psi_{\bla}(z) = \sum_{n=1}^\infty \alpha_n < \pi.
	$$
	Therefore, the function $\psi_{\bla}$ maps $\bC_+$ into $\bC_+$ and thus is a Herglotz function. Since
	$$   
	\psi_{\bla}(z)= 1+o(1), \qquad z\to \infty,
	$$
	and since $\psi_{\bla}$ assumes real values for 
	$x\in\bR\setminus \left(\{0\}\cup\{\lambda_{2n}\}_{n=1}^\infty \right)$, we conclude that
	$$    
	\psi_{\bla}(z) = 1 + \int_{\bR} \frac{d\nu(t)}{t-z}
	$$
	for some non-negative Borel measure $\nu$ with $\supp\nu\subset \{0\}\cup
	\{\lambda_{2n}\}_{n=1}^\infty $, i.e., $\nu\in\cM$.
	
	Conversely, consider an arbitrary $\nu \in \cM$ and denote by $\eta_k$, $k\in\bN$, the positive zeros of $\phi_\nu$ as explained above. Set $\lambda_{2n} = \xi_n$ and $\lambda_{2n-1}=\eta_n$; then $\bla:=(\lambda_n)_{n\in\bN}$ is an element of $\Lambda$, and we next prove that $\phi_\nu = \psi_{\bla}$. 
	
	As $\phi_\nu(\eta_1) = 0$, we see that 
	\[
	\frac{\phi_\nu(z)}{\eta_1 - z} = \frac{\phi_\nu(z) - \phi_\nu(\eta_1)}{\eta_1 - z} =
	\sum_{j=0}^\infty \frac{d_j'}{\xi_j - z},
	\]
	with $d'_j = d_j/(\eta_1-\xi_j)$. Therefore, 
	\[
	\frac{\xi_1 - z}{\eta_1 - z}\phi_\nu(z)
	=   \sum_{j=0}^\infty d_n' \frac{\xi_1 - z}{\xi_j - z}
	=  \sum_{j=0}^\infty d_j' \Bigl(1 + \frac{\xi_1 - \xi_j}{\xi_j - z}\Bigr) 
	= c + \sum_{j=0}^\infty \frac{d_j''}{\xi_j - z}
	\]
	where 
	\[
	c = \sum_{j=0}^\infty d_j' =  - \sum_{j=0}^\infty \frac{d_j}{\xi_j - \eta_1} = 1
	\]
	and 
	\[
	d_j'' = d_j \frac{\xi_1 - \xi_j}{\eta_1 - \xi_j} 
	\]
	with $d_1''=0$ and $d_j''>0$ for $j>0$. Therefore,
	\[
	\phi_{1,\nu}(z):=\frac{\xi_1 - z}{\eta_1 - z}\phi_\nu(z)
	\]
	is a Herglotz function. 
	
	In the same manner we show that the functions
	\[
	\phi_{n,\nu}(z) := \phi_\nu(z) \prod_{k=1}^n \frac{\xi_k - z}{\eta_k - z}
	\]
	are Herglotz, and, passing to the limit, we see that the function
	\[
	h(z) := \phi_\nu(z)/\psi_{\bla}(z)
	\]
	is Herglotz. The function $h$ is analytic in $\bC\setminus\{0\}$, assumes real values on $\bR\setminus\{0\}$, and 
	$$     h(z)=1+o(1) \qquad z\to \infty.
	$$
	As a result, it assumes the form
	$   h(z)=1 - \alpha/z
	$
	with $\alpha\ge 0$. By construction, $h$ has no zeros in the domain $\bC\setminus\{0\}$, whence $\alpha=0$, i.e.,  $\phi_\nu=\psi_{\bla}$. 
\end{proof}

\subsection{Blaschke products}
Denote by $\Lambda$ a subset of the space $\ell_1(\bN)$ consisting of all its sequences $\bm{\lambda}=(\lambda_n)_{n\in\bN}$ with the property that $\lambda_n \in\bC_+\cup\{0\}$ for every $n\in\bN$. The set $\Lambda$ becomes a topological space with inherited topology of~$\ell_1(\bN)$. For every $\bm{\lambda}\in \Lambda$, the corresponding Blaschke product
\begin{equation}\label{eq:pre.Blaschke}
B(z,\bm{\lambda}):=\prod_{n=1}^\infty \frac{z-\lambda_n}{z-\overline\lambda_n}, \qquad z\in\bC_+,
\end{equation}
converges absolutely in $\bC_+$ in view of the inequality 
\[
\Bigl|\frac{z-\lambda_n}{z-\overline\lambda_n} - 1 \Bigr| \le \frac{2|\lambda_n|}{|\opn{Im} z|}
\]
and defines there a holomorphic function. 
Observe also that $|B(z,\bm{\lambda})|\le1$ on $\bC_+$ due to the inequality $|z-\lambda_n| \le |z-\overline{\lambda}_n|$
holding for all $n\in\bN$ and all $z\in\bC_+$. Lemma~3.2 of~\cite{HM-traces} implies the following convergence result. 

\begin{lemma}\label{lem:pre.Blaschke}
	For every fixed $z_0 \in \bC_+$, the mapping 
	\[
	\Lambda \ni \bm{\lambda} \mapsto B(z_0, \bm{\lambda}) \in \bC
	\]
	is continuous.
\end{lemma}

\subsection{Auxiliary facts on self-adjoint operators}

In this appendix, we collect several auxiliary facts on self-adjoint operators. They are well known but we are at a loss on precise references. Assume that $H$ is a Hilbert space with scalar product $\langle \,\cdot\,,\,\cdot\,\rangle$ and denote by $\mathscr{B}(H)$ the algebra of all bounded operators in $H$ and by $\mathscr{B}_+(H)$ its subset set of all self-adjoint positive operators. 

Spectral theorem for self-adjoint operators along with the Lebesque dominated convergence theorem produce the following result.  

\begin{proposition}\label{pro:app.domain12}  
	Assume that $A$ is a positive operator in a Hilbert space~$H$ and $f\in H$ is such that
\begin{equation}\label{eq.41y}
	\lim_{t\to +0}\langle(tI+A)^{-1}f,f\rangle =:\alpha<\infty.
\end{equation}
Then $f$ belongs to the domain of $A^{-1/2}$ and $\|A^{-1/2}f\|^2=\alpha$.
Conversely, if $f\in\dom A^{-1/2}$, then
\begin{equation}\label{eq.42y}
\lim_{t\to +0}\langle(tI+A)^{-1}f,f\rangle=\lim_{t\to +0}\|(tI+A)^{-1/2}f\|^2 =\|A^{-1/2}f\|^2.
\end{equation}
\end{proposition}

\begin{proposition}[\!\!{\cite[Sect.~VI.2.6]{Kato}}]\label{pro:app.inverse}
	Assume that $A$ and $B$ are positive operators and $A\le B$; 
	then $\dom (A^{-1/2}) \subset \dom (B^{-1/2})$ and $\|B^{-1/2}f\|\le \|A^{-1/2}f\|$ for all $f \in \dom(A^{-1/2})$. In particular, the operators $B^{-1/2}A^{1/2}$ and $A^{1/2}B^{-1/2}$ can be extended by continuity to elements of~$\mathscr{B}(H)$ of norm at most~$1$.
\end{proposition}

\begin{lemma}\label{le.42y}  Assume that $A,B\in \mathscr{B}_+(H)$, $A\le B$ and
	$ A_\varepsilon:= A+\varepsilon I, \,\, B_\varepsilon:= B+\varepsilon I \,\,(\varepsilon>0).
	$
	Then there exist the strong limits
	\begin{equation}\label{eq.1y}
	\slim\limits_{\varepsilon\to+0}  A^{1/2}B^{-1/2}_\varepsilon =
	\slim\limits_{\varepsilon\to+0} A^{1/2}_\varepsilon B^{-1/2}_\varepsilon =(B^{-1/2}A^{1/2})^*,
	\end{equation}
	\begin{equation}\label{eq.2y}
	\slim\limits_{\varepsilon\to+0} B^{-1/2}_\varepsilon A^{1/2} = 
	\slim\limits_{\varepsilon\to+0}  B^{-1/2}_\varepsilon  A^{1/2}_\varepsilon = B^{-1/2}A^{1/2}.
	\end{equation}
\end{lemma}

\begin{proof}
	We start by observing that, due to the Banach--Steinhaus theorem, the product of two strongly convergent operator families is strongly convergent. 
	Applying the Lebesque dominated convergence theorem to the spectral representation of $B$, we find that
	\begin{equation}\label{eq.4y}
	\slim\limits_{\varepsilon\to+0}B_{\varepsilon}^{-1/2}B^{1/2}=I, \qquad
	\slim\limits_{\varepsilon\to+0} \varepsilon^{1/2}B^{-1/2}_\varepsilon=0, 
	\end{equation}
	so that, for all $f\in H$, we get
	\begin{equation*}
	\lim\limits_{\varepsilon\to+0}  A^{1/2}B^{-1/2}_\varepsilon B^{1/2}f=A^{1/2} f, \quad
	\lim\limits_{\varepsilon\to+0} A^{1/2}_\varepsilon B^{-1/2}_\varepsilon B^{1/2}f =A^{1/2}f.
	\end{equation*}
	Since by Proposition~\ref{pro:app.inverse} the operators $A^{1/2}B^{-1/2}_\varepsilon$ and $A^{1/2}_\varepsilon B^{-1/2}_\varepsilon$ are uniformly bounded and since the range of the operator $B^{1/2}$ is everywhere dense in~$H$, the above relations imply that the limits in~\eqref{eq.1y} exist and are equal to the closure of the operator $A^{1/2} B^{-1/2}$, which is equal to $( B^{-1/2}A^{1/2})^*$.

	We next prove~\eqref{eq.2y}. It follows from Proposition~\ref{pro:app.inverse} that the operator $B^{-1/2}A^{1/2}$ is bounded; taking into account~\eqref{eq.4y} and the equality
	\begin{equation*}
		B_{\varepsilon}^{-1/2}A^{1/2}=(B_{\varepsilon}^{-1/2}B^{1/2})(B^{-1/2}A^{1/2}), 
	\end{equation*}
	we conclude that $\slim_{\varepsilon\to+0} B^{-1/2}_\varepsilon A^{1/2} = B^{-1/2} A^{1/2}$, i.e., get the first limit of~\eqref{eq.2y}. 
	Next, denote by $P_\varepsilon $ the spectral projector for $A$ corresponding to the interval $[\varepsilon,\infty)$ and set $P'_A(\varepsilon) = I - P_\varepsilon$.
	Applying the Lebesque dominated convergence theorem to the spectral representation of $A$, we find that
	\begin{equation*}
	\slim\limits_{\varepsilon\to+0}A^{-1/2}P_\varepsilon A_{\varepsilon}^{1/2}
	=I, \qquad
	\slim\limits_{\varepsilon\to+0}\varepsilon^{-1/2} P'_\varepsilon A_{\varepsilon}^{1/2}=0.
	\end{equation*}
	Using the first limit of~\eqref{eq.2y}, the second limit of~\eqref{eq.4y}, and the equality 
	\begin{equation*}
	B_{\varepsilon}^{-1/2}A_{\varepsilon}^{1/2}=(B_{\varepsilon}^{-1/2}A^{1/2})
	(A^{-1/2}P_\varepsilon A_{\varepsilon}^{1/2}) +(\varepsilon^{1/2}B^{-1/2}_\varepsilon)(\varepsilon^{-1/2} P'_\varepsilon A_\varepsilon^{1/2}),
	\end{equation*}
	we conclude that~$\slim_{\varepsilon\to+0} B^{-1/2}_\varepsilon A_\varepsilon^{1/2} =B^{-1/2} A^{1/2}$. 
	The proof is complete.
\end{proof}

\begin{lemma}\label{le.43y}
Assume that  $(A_n)_{n\in\bN}$ is a sequence of bounded and boundedly invertible operators in a Hilbert space~$H$ converging to an operator $A\in\mathscr{B}(H)$ in the strong operator topology. If $\sup_{n\in\bN}\|A_n^{-1}\|<\infty$ and $\slim\limits_{n\to\infty} A^*_n=A^*$, then $A$ is boundedly invertible and $\slim\limits_{n\to\infty} A^{-1}_n=A^{-1}$.
\end{lemma}

\begin{proof} Under the assumptions of the lemma, set $\alpha:=\sup_{n\in\bN}\|A_n^{-1}\|$. Since for every
$f\in H$ and $n\in\bN$ we have
$$    \|f\| \le \|A_n^{-1}\|\|A_nf\|\le \alpha \|A_nf\| $$
and
$$
\|f\| \le \|(A^*_n)^{-1}\|\|A^*_nf\| \le  \alpha \|A^*_nf\|,
$$
we see that 
$$    \|Af\|, \|A^*f\|\ge \alpha^{-1}\|f\|,  \qquad f\in H,
$$
i.e., the operators $A$ and $A^*$ have the trivial nullspace, so that the operator $A$ is invertible. 
Now it follows that, for every $f\in H$,
\[
	\|(A_n^{-1} -A^{-1})f\| 
		= \|A_n^{-1}(A- A_n)A^{-1}f\| \le \alpha \|(A- A_n)A^{-1}f\| \to 0
\]
as $n\to\infty$, and the proof is complete.
\end{proof}

%

\end{document}